\documentclass[english]{amsart}
\usepackage{amsfonts,amssymb,amsmath,amsgen,amsthm}
\usepackage{hyperref,color}
\usepackage{pdfsync}
\usepackage{mathtools}

\makeatletter
\theoremstyle{plain}
\newtheorem{thm}{\protect\theoremname}
\theoremstyle{definition}
\newtheorem{defn}[thm]{\protect\definitionname}
\theoremstyle{remark}
\newtheorem{rem}[thm]{\protect\remarkname}
\theoremstyle{plain}
\newtheorem{lem}[thm]{\protect\lemmaname}
\theoremstyle{plain}
\newtheorem{prop}[thm]{\protect\propositionname}
\newtheorem{corollary}{Corollary}

\makeatother

\usepackage{babel}
\providecommand{\definitionname}{Definition}
\providecommand{\lemmaname}{Lemma}
\providecommand{\propositionname}{Proposition}
\providecommand{\remarkname}{Remark}
\providecommand{\theoremname}{Theorem}

\def\R{{\mathbf R}}
\def\N{{\mathbf N}}
\def\Z{{\mathbf Z}}
\def\d{{\partial}}
\def\eps{\varepsilon}
\DeclareMathOperator{\RE}{Re}
\DeclareMathOperator{\IM}{Im}
\DeclareMathOperator{\diver}{div}

\numberwithin{equation}{section}

\date\today

\title[Instrinsically hydrodynamic approach to multidimensional QHD]{An intrinsically hydrodynamic approach to multidimensional QHD systems}

\author{Paolo Antonelli}
\address{Gran Sasso Science Institute, viale Francesco Crispi, 7, 67100 L'Aquila}
\email{paolo.antonelli@gssi.it}

\author{Pierangelo Marcati}
\address{Gran Sasso Science Institute, viale Francesco Crispi, 7, 67100 L'Aquila}
\email{pierangelo.marcati@gssi.it}

\author{Hao Zheng}
\address{Chinese Academy of Science, Academy of Mathematics and System Science, Zhongguancun East Rd. 55, 100190, Beijing}
\email{zhenghao@amss.ac.cn}

\subjclass{Primary: 35Q40; Secondary: 76Y05, 35Q35, 82D50.}
\keywords{Quantum Hydrodynamics, global existence, stability, finite energy, chemical potential}

\begin{document}
	\maketitle
	\begin{abstract}
		In this paper we consider the multi-dimensional Quantum Hydrodynamics (QHD) system, by adopting an intrinsically hydrodynamic approach. 
		The present paper continues the analysis initiated in \cite{AMZ} where the one dimensional case has been investigated. Here we extend the analysis to the multi-dimensional problem, in particular by considering two physically relevant classes of solutions. First of all we consider two-dimensional initial data endowed with point vortices; by assuming the continuity of the mass density and a quantization rule for the vorticity, we are able to study the Cauchy problem and provide global finite energy weak solutions. The same result can be obtained also by considering spherically symmetric initial data in the multi-dimensional setting.
		For rough solutions with finite energy, we are able to provide suitable dispersive estimates, which also apply to a more general class of Euler-Korteweg equations.
		
		Moreover we are also able to show the sequential stability of weak solutions with positive density. Analogously to the one dimensional case, this is achieved through the a priori bounds given by a new functional, first introduced in \cite{AMZ}.
		
	\end{abstract}

	\section{Introduction}\label{sect:intro}
	The quantum hydrodynamic (QHD) system is used as a model in various contexts in physics and engineering, such as superfluidity \cite{L}, Bose-Einstein condensation \cite{DGPS}, quantum plasmas \cite{HaasB}, or semiconductor devices \cite{Gar}. In general such system describes phenomena whose macroscopic characterization still exhibit quantum features.
	From the mathematical point of view it is an Euler type system for compressible fluids,  where the current density balance equation includes a third order term provided by a quantum potential encoding fluid macroscopic properties due to quantum interactions.
	The QHD system is given by
	\begin{equation}\label{eq:QHD_md}
		\left\{\begin{aligned}
			&\d_t\rho+\diver J=0\\
			&\d_tJ+\diver\left(\frac{J\otimes J}{\rho}\right)+\nabla p(\rho)=
			\frac12\rho\nabla\left(\frac{\Delta\sqrt{\rho}}{\sqrt{\rho}}\right),
		\end{aligned}\right.
	\end{equation}
	with $(t, x)\in\R\times\R^d$, where the unknowns $\rho$ and $J$ describe the fluid particle and momentum densities, respectively. The pressure term $p(\rho)$ for convenience will be assumed to satisfy a power law, namely $p(\rho)=\frac{\gamma-1}{\gamma}\rho^\gamma$, with $1<\gamma<\infty$ for $d=2$ and $1<\gamma<3$ for $d=3$.
	More general pressure laws can be taken into consideration and for instance it is also possible to treat non-monotone cases, see Remark \ref{rmk:vdW} for more details.
	\newline
	The total energy of the system, formally conserved along the flow of solutions, is given by
	\begin{equation}\label{eq:en_QHD_md}
		E(t)=\int_{\R^d}\frac{1}{2}|\nabla\sqrt{\rho}|^2+\frac12\frac{|J|^2}{\rho}+f(\rho)\,dx,
	\end{equation}
	where the internal energy $f(\rho)$ is related to the pressure term by the relation 
	$f(\rho)=\rho\int_0^\rho\frac{p(s)}{s^2}\,ds$.
	\newline
	The natural bounds given by \eqref{eq:en_QHD_md} imply that the only available control for the velocity field $v$, defined by $J=\rho v$, is in $L^2$, with respect to the measure $\rho\,dx$. In particular, there is no control of the velocity field in the vacuum region. For this reason in order to deal with finite energy weak solutions to \eqref{eq:QHD_md}, it is more convenient to consider the unknowns 
	$(\sqrt{\rho}, \Lambda)$, which define the hydrodynamic variables by $\rho=(\sqrt{\rho})^2$, $J=\sqrt{\rho}\Lambda$, see Definition \ref{def:FEWS} below for more details. In this paper we will refer to $(\rho,J)$ as the macroscopic hydrodynamic variables solving the system \eqref{eq:QHD_md}, and the pair $(\sqrt\rho,\Lambda)$ will often be referred as the hydrodynamic state associated to $(\rho, J)$. 
	\\
	The lack of suitable a priori bounds prevents the study of solutions to \eqref{eq:QHD_md} by using compactness arguments like it is done for viscous systems, see for example \cite{AS1, AS2, LV, LZZ} where a viscous counterpart of system \eqref{eq:QHD_md} is considered, see also \cite{AS3, AS4} where a similar system is studied by using a suitable truncation argument. On the contrary, for the QHD system most of the existing results in the literature are of perturbative type \cite{JMR, LM, HLM, HLMO, AMZ_q}.
	\\
	On the other hand, the QHD system \eqref{eq:QHD_md} enjoys a useful relation with nonlinear Schr\"odinger equations, established through the Madelung transformations \cite{Mad}. More specifically, given a wave function it is possible to define suitable hydrodynamical variables, in such a way that a solution to the NLS equation determines a solution to the QHD system by means of the Madelung transform. This relation was rigorously set up and exploited in \cite{AM1, AM2} in order to prove the global existence of finite energy weak solutions to \eqref{eq:QHD_md} without any smallness or regularity assumptions. In particular, in \cite{AM1, AM2} the authors developed a polar factorization technique which overcome the difficulty of defining the velocity field in the vacuum region and allows to define the state $(\sqrt{\rho}, \Lambda)$ almost everywhere. The main drawback of this approach is that the initial data for \eqref{eq:QHD_md} need to be consistent with some given wave function, in order to define an initial datum for the wave function dynamics.
	\newline
	In \cite{AMZ} we overcome this problem in the one dimensional case by developing a wave function lifting method, namely for a large class of hydrodynamic states $(\sqrt{\rho}, \Lambda)$, it is possible to determine an associated wave function. An alternative approach was proposed in \cite{Bian}, where the problem is reduced to exact integrability conditions for cotangent vector fields. By assuming a quantization condition for the vorticity, it is solved in metric measure spaces satisfying some specific assumptions.
	By using the wave function lifting it is then possible to show the existence of global in time finite energy weak solutions to the one dimensional QHD system, without the assumption that the initial data are generated by a wave function. Moreover, by introducing a new functional which, at least formally, is bounded over compact time intervals, it is also possible to determine a class of solutions enjoying a stability property. More precisely, by assuming the uniform bounds given by the total energy and the new functional, it is possible to prove that any sequence of solutions enjoying the previous bounds admits  a converging subsequence and the limit is a finite energy weak solution to the one dimensional QHD system. The stability result given in \cite{AMZ} holds for a general class of solutions satisfying the aforementioned uniform bounds, and are not necessarily restricted to solutions generated by NLS equations.
	\newline
	The purpose of the present paper is to continue the analysis initiated in \cite{AMZ}, in order to develop an intrinsically hydrodynamic theory to QHD systems in multidimensional cases. 
	It is straightforward to see that considering the problem for $d\geq2$ entails major mathematical difficulties.
	If we want to follow a similar strategy as in \cite{AMZ}, one of the main obstacles is the characterization of the vacuum region $\{\rho=0\}$. Indeed, while in one space dimension the energy estimate and the Sobolev embedding yield the mass density to be continuous, which consequently implies that $\{\rho=0\}$ is a closed set, in the multidimensional case we loose this topological information and in general it is not possible to characterize the structure of the vacuum region with the same methods of \cite{AMZ}. 
	Moreover, for $d\geq2$ an extra structural property of solutions needs to be taken into account, namely the generalized irrotationality condition introduced in \cite{AM1, AM2}, see also Definition \ref{def:FEWS} and the following Remark \ref{rmk:vort} below. Heuristically, the generalized irrotationality condition imposes to consider solutions such that the velocity field is irrotational in the region $\{\rho>0\}$, whereas in the vacuum region 
	$\{\rho=0\}$ vortices may appear. This condition is dictated by physical observations and numerical simulations for Bose-Einstein condensates and superfluid Helium II. Moreover, theoretical predictions 
	\cite{Ons, Fey} say that the vorticity of a quantum fluid must be quantized, namely the circulation of the velocity field on a closed curve around a vortex must be an integer multiple of the vorticity quantum. 
	This is consistent with the formal analogy given by the Madelung transform. Indeed given a wave function $\psi$, by expressing it in terms of its amplitude and phase $\psi=\sqrt{\rho}e^{iS}$, it follows that the velocity field $v$ is given by the gradient of the phase, $v=\nabla S$. In this way the velocity field is straightforwardly irrotational, however the formal description given by the Madelung transform fails in the vacuum region, where the phase may not be defined. In \cite{AM1, AM2} the authors overcome the formal analogy given by the Madelung transformation by means of the polar decomposition approach. In this way it is possible to consider hydrodynamic states $(\sqrt{\rho}, \Lambda)$ which are defined almost everywhere and satisfy a generalized irrotationality condition. 
	\\
	In the present paper we are able to consider a class of initial data including a countable number of non-accumulating quantized vortices and such that the mass density is continuous. In this framework we are able to build up a suitable wave function lifting method (see Section \ref{sect:lift_2d}), which then yields the existence of global in time finite energy weak solutions to \eqref{eq:QHD_md}. This result also allows to deal in a similar way with a countable number of vortex lines on a three dimensional framework. Moreover we are also able to study another physically relevant case, namely by assuming radial symmetry of the initial data. Also in this case it is possible to establish a wave function lifting result for hydrodynamical functions and hence to infer global existence of finite energy weak solutions to \eqref{eq:QHD_md}.
	\\
	Moreover, in the present paper we also provide a class of purely hydrodynamical dispersive estimates for the QHD system; actually some of them can also be stated for general Euler-Korteweg systems, in particular in Proposition \ref{prop:mor_md} we prove a Morawetz-type estimate.
	\\
	Analogously to the one dimensional case studied in \cite{AMZ}, in this paper we also show a stability result for weak solutions to \eqref{eq:QHD_md}, however the full generality of those results cannot be achieved and we focus our analysis to the case where we assume the positivity for the whole sequence of the mass densities. 
	
	In order to establish a stability theory for finite energy weak solutions to fluid systems, it is necessary to derive suitable a priori bounds, often provided by diffusive or dissipative properties.
	It is well known that conservative systems, like \eqref{eq:QHD_md}, do not enjoy such structural stability property.
	In particular this is in sharp contrast with viscous systems with constant viscosity (see for instance \cite{Fei, Lions}), where the energy dissipation provides a control in Sobolev norms for the velocity field. Also in the case of degenerate viscosity \cite{LLX, LX}, where the energy dissipation does not yield sufficient bounds due to the degeneracy in the vacuum region, further a priori estimates are still available \cite{BrD, BDL}. This fact allows to infer a compactness property for weak solutions, see for instance \cite{AS1} for the QNS equations, namely system \eqref{eq:QHD_md} augmented by a degenerate viscosity, or \cite{AS3} for the analysis of a particular Navier-Stokes-Korteweg system. Unfortunately, none of the above methods can be used for QHD or Euler-Korteweg type systems, therefore we will have to extract the information from other physical properties. For instance, we will take into consideration the properties of the chemical potential.
	
	As we already mentioned before, one of the main novelties of our analysis relies on the introduction of a new functional which provides suitable a priori bounds for solutions to \eqref{eq:QHD_md}.
	Formally, this higher order functional gives a control in $L^2(\rho\,dx)$ of the chemical potential associated to system \eqref{eq:QHD_md}, which is defined by
	\begin{equation}\label{eq:chem_md}
		\mu=\frac{\delta E}{\delta\rho}=-\frac12\frac{\Delta\sqrt{\rho}}{\sqrt{\rho}}+\frac12|v|^2+f'(\rho),
	\end{equation}
	see for example \cite{GLT}. Unfortunately the chemical potential $\mu$ cannot be used to carry out a satisfactory mathematical analysis in the framework of weak solutions to \eqref{eq:QHD_md}. For this reason it will be more convenient to consider
	\begin{equation}\label{eq:lambda_md_intro}
		\xi:=\rho\mu=-\frac14\Delta\rho+e+p(\rho),
	\end{equation}
	where 
	\begin{equation}\label{eq:endens_md}
		e=\frac12|\nabla\sqrt{\rho}|^2+\frac12|\Lambda|^2+f(\rho)
	\end{equation}
	is the energy density. In this way we define
	\begin{equation}\label{eq:higher_md}
		I(t)=\int\frac{\xi^2}{\rho}+(\d_t\sqrt{\rho})^2\,dx.
	\end{equation}
	More precisely, by introducing a "generalized chemical potential"
	\begin{equation}\label{eq:lambda_intro}
		\lambda=\frac{\xi}{\sqrt{\rho}}\chi_{\{\rho>0\}},
	\end{equation}
	we can write $I(t)=\int\lambda^2+(\d_t\sqrt{\rho})^2\,dx$.
	A more detailed discussion about $\lambda$ and $I(t)$ will be given later, see also Section 7 in \cite{AMZ}.
	\\
	We can now present the main results of our paper.
	\subsection{Global existence of weak solutions for initial data with point vortices}
	
	The first result we show deals with the existence of global in time finite energy weak solutions.
	As mentioned before we are going to exploit a wave function lifting argument in the two dimensional case by considering initial data with point vortices. More precisely we consider initial data $(\rho_0(x),J_0(x))$ with $x\in\R^2$ such that $\rho_0$ is continuous, and the vacuum set $V=\{x:\,\rho_0(x)=0\}$ consists of isolated points, 
	\begin{equation}\label{eq:vac_intro}
		V=\{x_{(\alpha)}\}_{\alpha\in\mathcal{A}}\subset\R^2,
	\end{equation}
	where $\mathcal{A}$ is an at most countable set of indices, and in the case that $\mathcal{A}$ is an infinite set, we further require $\underset{\alpha\ne\beta}{\inf} |x_{(\alpha)}-x_{(\beta)}|>0$. On the other hand, let $v_0=J_0/\rho_0$ be the initial velocity field, which is well-defined almost everywhere, then we say $(\rho_0,J_0)$ satisfies the \emph{quantised vorticity condition} if 
	\begin{equation}\label{eq:QV_intro}
		\begin{cases}
			v_0\in \mathcal M(\R^2)\\
			\nabla\wedge v_0=2\pi \underset{\alpha\in\mathcal{A}}{\sum}k_\alpha\delta_{x_{(\alpha)}}, \hspace{0.3cm} k_\alpha\in\Z,
		\end{cases}
	\end{equation}
	where $\mathcal M(\R^2)$ is the space of Radon measures and $\delta_{x_{(\alpha)}}$ is the Dirac-delta distribution supported at $x_{(\alpha)}$ for $\alpha\in\mathcal{A}$. 
	
	The condition \eqref{eq:QV_intro} is dictated by physical motivations. Indeed as we discussed before, one of the main features of quantum fluids is the appearance of quantized vortices \cite{Fey, Wiem, Tsub}. More precisely, quantum fluids are characterized by the fact that the flow is irrotational in the region $\{\rho>0\}$ and on the vacuum region quantized vortices may appear. This means that if $x_0\in \R^2$ is a vacuum point, then the circulation of the velocity field along a closed curve around $x_0$ must be quantized,
	\begin{equation}\label{eq:108}
		\oint_{\gamma} v(x)\cdot d\overrightarrow{l}(x)=2\pi k
	\end{equation}
	for some $k\in\mathbb{Z}$. By continuity, the integer $k$ is uniquely
	determined by the homotopy class of curve $\gamma$. Formally by Stokes' theorem, the quantised vorticity condition \eqref{eq:QV_intro} is the distributional formulation of identity \eqref{eq:108}. Notice that the latter identity is the well known Bohr-Sommerfeld quantization rule used in the old quantum theory, see \cite[Paragraph 48]{LL3}.
	
	\begin{thm}[Global Existence of finite energy weak solutions]\label{thm:glob_md}
		Let us consider a hydrodynamic state $(\sqrt{\rho_0}, \Lambda_0)$ satisfying the bound
		\begin{equation}\label{eq:C1_md_intro}
			\|\sqrt{\rho_0}\|_{H^1(\R^2)}+\|\Lambda_0\|_{L^2(\R^2)}\leq M_1,
		\end{equation}
		for some $M_1>0$ 
		and let us further assume that $\rho_0$ is continuous, with isolated vacuum points, namely there exists an at most countable index set $\mathcal{A}$, such that 
		\begin{equation}\label{eq:vac_thm}
			V=\{\rho_0=0\}=\{x_{(\alpha)}\}_{\alpha\in\mathcal{A}}\subset\R^2,\quad \inf_{\alpha\ne \beta}|x_{(\alpha)}-x_{(\beta)}|\geq \alpha>0
		\end{equation}
		and $v_0=J_0/\rho_0$ satisfies the quantized vorticity condition 
		\begin{equation}\label{eq:QV_thm}
			\begin{cases}
				v_0\in \mathcal M(\R^2)\\
				\nabla\wedge v_0=2\pi \underset{\alpha\in\mathcal{A}}{\sum}k_j\delta_{x_{(\alpha)}}, \hspace{0.3cm} k_\alpha\in\Z.
			\end{cases}
		\end{equation}
		Then there exists a global in time finite energy weak solution to the Cauchy problem \eqref{eq:QHD_md} which conserves the total energy for all time, namely $E(t)=E(0)$. In particular we have
		\begin{equation}
			\|\sqrt{\rho}\|_{L^\infty(0, T;H^1(\R^2))}+\|\Lambda\|_{L^\infty(0, T;L^2(\R^2))}\leq C(M_1),
		\end{equation}
		for some $C(M_1)>0$.
	\end{thm}
	\begin{rem}
		As we will see in Proposition \ref{prop:lift_2d} below, the assumptions on the initial data in Theorem \ref{thm:glob_md} implies the existence of a finite energy wave function $\psi_0\in H^1$, associated to the hydrodynamic state $(\sqrt{\rho_0}, \Lambda_0)$. By evolving it through a suitable nonlinear Schr\"odinger equation and by exploiting the polar factorization, see Lemma \ref{lemma:polar} below, we are able then to show the existence of a finite energy weak solution, as stated in Theorem \ref{thm:glob_md} above.
		\newline
		Moreover, let us remark that, as a byproduct of our wave function lifting, we also obtain that the quantized vorticity condition \eqref{eq:QV_thm} implies the generalized irrotationality condition. We provide more details about this fact in Remark \ref{rmk:qv_gic} below.
	\end{rem}
	
	Furthermore, we are able to construct a family of more regular weak solutions. As we will see, the result we establish is not given in terms of the Sobolev regularity for $(\sqrt{\rho}, \Lambda)$ but rather in terms of the functions involved in the analysis of the relavant functional \eqref{eq:higher_md}.
	
	\begin{thm}[Global existence of regular weak solutions]\label{thm:glob2_md}
		Let us consider a finite energy hydrodynamic state $(\sqrt{\rho_0}, \Lambda_0)$ satisfying the same hypotheses as Theorem \ref{thm:glob_md}. Let us further assume that 
		\begin{align}
			&\Delta\rho_0, \nabla J_0\in L^1_{loc}(\R^2),\label{eq:1.15}\\
			\label{eq:C2_md_intro}
			&\|\frac{\Lambda_{0,j}^2}{\sqrt{\rho_0}}-\d_{x_j}^2\sqrt{\rho_0}\|_{L^2(\R^2)}
			+\|\frac{\d_{x_j} J_{0,j}}{\sqrt{\rho_0}}\|_{L^2(\R^2)}\leq M_2,\quad j=1,2.
		\end{align}
		Then the solution to \eqref{eq:QHD_md}, constructed in Theorem \ref{thm:glob_md}, enjoys the following properties: 
		\begin{itemize}
			\item
			there exists $\lambda\in L^\infty(0, T; L^2(\R^2))$, such that $\lambda=0$ a.e. on $\{\rho=0\}$ and
			\begin{equation}\label{eq:rho_mu_md}
				\sqrt{\rho}\lambda=-\frac14\triangle\rho+e+p(\rho),
			\end{equation}
			where $e$ denotes the total energy density
			\begin{equation*}
				e=\frac12|\nabla\sqrt{\rho}|^2+\frac12|\Lambda|^2+f(\rho);
			\end{equation*}
			\item for any $0<T<\infty$ we have
			\begin{equation}\label{eq:lambda_md}
				\|\d_t\sqrt{\rho}\|_{L^\infty(0, T;L^2(\R^2))}+\|\lambda\|_{L^\infty(0, T;L^2(\R^2))}\leq C(T, M_1, M_2).
			\end{equation}
		\end{itemize}
		Furthermore, for these solutions the following bounds hold true
		\begin{equation*}
			\|\rho\|_{L^\infty(0, T;H^2(\R^2))}+\|J\|_{L^\infty(0, T;H^1(\R^2))}+\|\sqrt{e}\|_{L^\infty(0, T;H^1(\R^2))}\leq C(T, M_1, M_2).
		\end{equation*}
	\end{thm}
	\vspace{0.1cm}
	\begin{rem}\label{rmk:C2_md}
		Since in Theorem \ref{thm:glob2_md} we assume the initial data $(\rho_0,J_0)$ have point vacuum and satisfy \eqref{eq:1.15}, the functions in condition \eqref{eq:C2_md_intro} are well-defined almost everywhere on $\R^2$ for $j=1,2$. The condition \eqref{eq:C2_md_intro} plays an essential role in characterising the higher order regularity of the lifted wave function given by Proposition \ref{prop:lift_2d}. We also remark that these conditions are sufficient and necessary to construct an initial wave function associated to hydrodynamic data with pointwise vacuum, and the bounds $\eqref{eq:C2_md_intro}$ are preserved by the solutions we obtained in Theorem \ref{thm:glob2_md}.
	\end{rem}
	\begin{rem}
		We remark that for the solutions constructed in Theorem \ref{thm:glob2_md}, the higher order functional in \eqref{eq:higher_md} is well defined for almost every time, as the quantity $\lambda$ constructed in the Theorem above matches the definition in \eqref{eq:lambda_intro}. Moreover, from \eqref{eq:lambda_md} we also infer that $I(t)$ is uniformly bounded over compact time intervals. 
		\newline
		As we will see later, the same holds true also for the solutions constructed in Theorems \ref{thm:glob2_3d} and \ref{thm:glob2_s}.
	\end{rem}
	
	The argument of Theorem \ref{thm:glob_md} and Theorem \ref{thm:glob2_md} also applies to typical 3D initial data. For example we can consider the 3D initial hydrodynamic state $(\sqrt{\rho_0}(x),\Lambda_0(x))$, $x=(x_1,x_2,x_3)\in\R^3$, with planar symmetry given as follows. By considering a 2D finite energy hydrodynamic state $(\sqrt{\rho_1}, \Lambda_1)$ and $\sqrt{\rho_2}\in H^1(\R)$, with $\sqrt{\rho_2}\geq0$, we define
	\begin{equation}\label{eq:cyldata_intro}
		\sqrt{\rho_0}(x)=\sqrt{\rho_1}(x_1,x_2)\sqrt{\rho_2}(x_3)\quad and \quad \Lambda_0(x)=\left(\Lambda_1(x_1,x_2)\sqrt{\rho_2}(x_3),0\right).
	\end{equation}
	We assume $(\rho_1,J_1)$ has point vortices, namely $\rho_1$ is continuous with vacuum structure \eqref{eq:vac_intro}, and the velocity field $v_1=J_1/\rho_1$ satisfies the quantised vorticity condition \eqref{eq:QV_intro}. In this case the vortex filaments of $(\rho_0(x),J_0(x))$ appear as parallel straight lines in $x_3$ direction, which is a widely studied vortex structure of quantum flows in 3D space \cite{BBS}. For such initial data, we can still prove the global existence of weak solutions, both with finite energy and in the regular class of the functional $I(t)$.
	
	\begin{thm}\label{thm:glob_3d}
		Let $(\sqrt{\rho_0}, \Lambda_0)$ be an initial hydrodynamic state given as in \eqref{eq:cyldata_intro}, where $(\sqrt{\rho_1},\Lambda_1)$ satisfies the vorticity structure \eqref{eq:vac_intro}, \eqref{eq:QV_intro} and the bound 
		\begin{equation*}
			\|\sqrt{\rho_1}\|_{H^1(\R^2)}+\|\Lambda_1\|_{L^2(\R^2)}\leq M_1,
		\end{equation*}
		for some $M_1>0$. Moreover we assume that $\sqrt{\rho_2}\in H^1(\R)$. Then there exists a global in time finite energy weak solution to the Cauchy problem \eqref{eq:QHD_md}, which conserves the total energy for all times, namely $E(t)=E(0)$. In particular we have
		\begin{equation}
			\|\sqrt{\rho}\|_{L^\infty(0, T;H^1(\R^3))}+\|\Lambda\|_{L^\infty(0, T;L^2(\R^3))}\leq C(M_1),
		\end{equation}
		for some $C(M_1)>0$.
	\end{thm}
	
	\begin{thm}\label{thm:glob2_3d}
		Furthermore, let us assume 
		\begin{equation*}
			\|\frac{\Lambda_{1,j}^2}{\sqrt{\rho_1}}-\d_{x_j}^2\sqrt{\rho_1}\|_{L^2(\R^2)}
			+\|\frac{\d_{x_j} J_{1,j}}{\sqrt{\rho_1}}\|_{L^2(\R^2)}\leq M_2,\quad j=1,2,
		\end{equation*}
		and $\sqrt{\rho_2}\in H^2(\R)$. Then there exists $\lambda\in L^\infty(0, T; L^2(\R^3))$, such that $\lambda=0$ a.e. on $\{\rho=0\}$,
		\begin{equation*}
			\sqrt{\rho}\lambda=-\frac14\triangle\rho+e+\rho f'(\rho)
		\end{equation*}
		and for any $0<T<\infty$ we have
		\begin{equation*}
			\|\d_t\sqrt{\rho}\|_{L^\infty(0, T;L^2(\R^3))}^2+\|\lambda\|_{L^\infty(0, T;L^2(\R^3))}\leq C(T, M_1, M_2).
		\end{equation*}
		Furthermore, for such solutions the following bounds hold true
		\begin{equation*}
			\|\rho\|_{L^\infty(0, T;H^2(\R^3))}+\|J\|_{L^\infty(0, T;H^1(\R^3))}+\|\sqrt{e}\|_{L^\infty(0, T;H^1(\R^3))}\leq C(T, M_1, M_2).
		\end{equation*}
	\end{thm}
	\vspace{0.1cm}
	
	\subsection{Global existence of spherically symmetric weak solutions}
	
	Now we consider the spherically symmetric hydrodynamic state in the sense that
	\begin{equation*}
		\sqrt{\rho_0}(x)=\sqrt{\rho_0}(r)\quad and\quad \Lambda_0(x)=\Lambda_0(r)\frac{x}{|x|},
	\end{equation*}
	where $x\in\R^d$, $r=|x|$ and $\Lambda_0(r)\in \R$. With suitable regularity assumption on $(\sqrt{\rho_0},\Lambda_0)$, we can prove the global existence of spherically symmetric weak solutions to the QHD system in the finite energy space and in the regular class of the functional \eqref{eq:higher_md}.
	
	\begin{thm}\label{thm:glob_s}
		Let $(\rho_0, J_0)$ be a spherically symmetric initial data given by the hydrodynamic state  
		\begin{equation*}
			\sqrt{\rho_0}(x)=\sqrt{\rho_0}(r)\quad and\quad \Lambda_0(x)=\Lambda_0(r)\frac{x}{|x|}
		\end{equation*}
		such that $\Lambda_0=0$ a.e. on the set $\{\rho_0=0\}$. Let us further assume the bounds
		\begin{equation}\label{eq:C1_s_intro}
			\|\sqrt\rho_0\|_{H^1(\R^d)}+\|\Lambda_0\|_{L^2(\R^d)}\le M_1,
		\end{equation}
		for some $M_1>0$.
		Then there exists a spherically symmetric global in time finite energy weak solution to the Cauchy problem \eqref{eq:QHD_md} which conserves the total energy for all times, i.e. $E(t)=E(0)$. In particular we have
		\begin{equation*}
			\|\sqrt{\rho}\|_{L^\infty(0, T;H^1(\R^d))}+\|\Lambda\|_{L^\infty(0, T;L^2(\R^d))}\leq C(M_1),
		\end{equation*}
		for some $C(M_1)>0$.
	\end{thm}
	Analogously to the case with point vortices, also here we provide an existence result for regular weak solutions.
	\begin{thm}\label{thm:glob2_s}
		Let $(\rho_0, J_0)$ be a pair of spherically symmetric initial data satisfying the hypotheses of Theorem \ref{thm:glob_s}. We further assume that the initial data satisfy the following conditions:
		\begin{itemize}
			\item $\triangle\rho_0\in L^1_{loc}(\R^d)$, $\diver J_0\in L^1_{loc}(\R^d)$;
			\item the initial energy density $e_0:=\frac12|\nabla\sqrt{\rho_0}|^2+\frac12|\Lambda_0|^2+f(\rho_0)$ is continuous;
			\item the initial data satisfy the bounds
			\begin{equation}\label{eq:C2_s_intro}
				\|(\frac{|\Lambda_0|^2}{\sqrt{\rho_0}}-\triangle\sqrt{\rho_0})\mathbf{1}_{\{\rho_0>0\}}\|_{L^2}
				+\|\frac{\diver J_0}{\sqrt{\rho_0}}\mathbf{1}_{\{\rho_0>0\}}\|_{L^2}\leq M_2.
			\end{equation}
		\end{itemize}
		Then there exists a spherically symmetric $\lambda\in L^\infty(0, T; L^2(\R^d))$, such that $\lambda=0$ a.e. on $\{\rho=0\}$ and
		\begin{equation}\label{eq:rho_mu_s}
			\sqrt{\rho}\lambda=-\frac14\triangle\rho+e+p(\rho),
		\end{equation}
		where $e$ denotes the total energy density
		\begin{equation*}
			e=\frac12|\nabla\sqrt{\rho}|^2+\frac12|\Lambda|^2+f(\rho).
		\end{equation*}
		Moreover, for any $0<T<\infty$ we have
		\begin{equation}
			\|\d_t\sqrt{\rho}\|_{L^\infty(0, T;L^2(\R^d))}+\|\lambda\|_{L^\infty(0, T;L^2(\R^d))}\leq C(T, M_1, M_2).
		\end{equation}
		Furthermore, for such solutions the following bounds hold
		\begin{equation*}
			\|\rho\|_{L^\infty(0, T;H^2(\R^d))}+\|J\|_{L^\infty(0, T;H^1(\R^d))}+\|\sqrt{e}\|_{L^\infty(0, T;H^1(\R^d))}\leq C(T, M_1, M_2).
		\end{equation*}
	\end{thm}
	\vspace{0.1cm}
	
	\subsection{Dispersive estimates}
	It is well known that the nonlinear Schr\"odinger equation enjoys various types of dispersive estimates, see for instance \cite{GV, GVQ, Tao, Caz}. 
	In particular, most of the results concerning Morawetz and virial type estimates, show a strong connection with the underlying hydrodynamic structure associated to the nonlinear Schr\"odinger equation. The purpose of the following results is to provide intrinsically hydrodynamic dispersive estimates for the QHD system.

	\begin{thm}[Morawetz-type and dispersive estimates]\label{thm:disp}
		Let $(\rho, J)$ be a finite energy weak solution to \eqref{eq:QHD_md} satisfying the bound
		\[
		\|\sqrt{\rho}\|_{L^\infty(0, T;H^1(\R^d))}+\|\Lambda\|_{L^\infty(0, T;L^2(\R^d))}\leq M_1.
		\]
		\begin{itemize}
			\item[(i)] We have
			\begin{equation}\label{eq:disp_1}
				\||\nabla|^{\frac{3-d}{2}}\rho\|_{L^2(\R_t\times\R^d_x)}^2
				+\||\nabla|^{\frac{1-d}{2}}\sqrt{\rho p(\rho)}\|_{L^2(\R_t\times\R^d_x)}^2\lesssim M_1^4.
			\end{equation}
			\item[(ii)] If we further assume that $\int_{\R^d} |x|^2\rho_0(x)\,dx<\infty$ at initial time, and that the total energy $E(t)$, given by \eqref{eq:en_QHD_md} is non-increasing, then
			\begin{equation}\label{eq:disp_2}
				\|\nabla\sqrt{\rho}(t)\|_{L^2}+\|\Lambda(t)-\frac{x}{t}\sqrt{\rho}(t)\|_{L^2}+\|\rho^\frac{\gamma}{2}(t)\|_{L^2}\lesssim t^{-\sigma},
			\end{equation}
			where $\sigma=\min\{1, \frac{d}{2}(\gamma-1)\}$.
		\end{itemize}
	\end{thm}
	
	\begin{rem}
		The dispersive estimates \eqref{eq:disp_1} and \eqref{eq:disp_2} hold for general weak solutions to the QHD system with finite energy. We emphasize that actually such estimates are also valid for a more general class of Euler-Korteweg type equations, as stated in Proposition \ref{prop:mor_md}.
	\end{rem}
	
	\subsection{Stability}
	
	The last result we present in this paper is the compactness property for the class of weak solutions satisfying the uniform bounds given by the energy and the functional $I(t)$. Here we need some additional assumptions on the sequence of solutions, for example the positivity of mass densities, or the spherical symmetry of solutions. Nevertheless let us remark we do not need a uniform lower bound on the mass density, and also our result below holds true for a general class of solutions, not necessarily those one constructed in the existence theorems given above.
	
	\begin{thm}[Stability]\label{thm:stab_md}
		Let us assume $\{(\rho_n, J_n)\}_{n\geq1}$ is a sequence of weak solutions to \eqref{eq:QHD_md} satisfying the uniform bounds 
		\begin{equation}\label{eq:unif1_md_intro}
			\begin{aligned}
				&\|\sqrt{\rho_n}\|_{L^\infty(0, T;H^1(\R^d))}+\|\Lambda_n\|_{L^\infty(0, T;L^2(\R^d))}\leq M_1,\\
				&\|\d_t\sqrt{\rho_n}\|_{L^{\infty}(0, T;L^2(\R^d))}+\|\lambda_n\|_{L^\infty(0, T;L^2(\R^d))}\leq M_2,
			\end{aligned}
		\end{equation}
		where $\lambda_n$ is implicitly given by the identity \eqref{eq:rho_mu_md}. 
		Let us also assume at least one of the following conditions for all $n\in\N$ and almost every $t\in [0,T]$ holds
		\begin{itemize}
			\item[(1)] $\rho_n(t,\cdot)$ is continuous with $\rho_n(t,\cdot)>0$;
			\item[(2)] $(\rho_n, J_n)$ is a spherically symmetric solution, and the total energy density $e_n(t, \cdot)=\frac12|\nabla\sqrt\rho_n|^2+\frac12|\Lambda_n|^2+f(\rho_n)$ is continuous.
		\end{itemize}
		Then up to subsequences, we have that
		\begin{equation*}
			\begin{aligned}
				\sqrt{\rho_n}\to&\sqrt{\rho}&\textrm{in}\;L^p(0, T;H^1_{loc}(\R^d)),\\
				\Lambda_n\to&\Lambda&\textrm{in}\;L^p(0, T;L^2_{loc}(\R^d)),
			\end{aligned}
		\end{equation*}
		where $(\sqrt{\rho}, \Lambda)$ is a finite energy weak solution to \eqref{eq:QHD_md} and $2\le p<\infty$.
	\end{thm}
	
	We conclude this introduction by discussing some existing mathematical literature.
	The QHD system, together with similar evolutionary PDE models, is widely studied. Previous results can be found in \cite{JMR, LM, GM, JLM, HLM, HLMO}. System \eqref{eq:QHD_md} is also intimately related to the class of Euler-Korteweg fluids \cite{BG}, encoding capillary effects in the description; local and global analysis of small, regular perturbations of constant states are discussed in \cite{BGDD, AH}. The method of convex integration can be applied also to Euler-Korteweg system in order to show the existence of infinitely many solutions emanating from the same initial data \cite{DFM}, see also \cite{MS} where a non-uniqueness result for system \eqref{eq:QHD_md} is given by adopting a different strategy, related to the underlying wave function dynamics.
	Recently also a class of viscous quantum fluid dynamical systems was considered. Such models can be derived from the Wigner-Fokker-Planck equation \cite{JLMG}, see also the review \cite{J}. The analysis of finite energy weak solutions for the quantum Navier-Stokes system was done in \cite{JqNS, AS1, AS2, LV}, see also \cite{LLX, LX} where similar arguments are used to study the compressible Navier-Stokes system with degenerate viscosity.
	
	We structure the contents of this paper as follows. Some notations and preliminary results that will be used through this paper are given in Section \ref{sect:prel}. The wave function lifting method, together with Theorem \ref{thm:glob_md} and Theorem \ref{thm:glob2_md}, for hydrodynamic data with pointwise vacuum and quantised vorticity is proved in Section \ref{sect:lift_2d}. On the other hand in Section \ref{sect:lift_s} we introduce the spherically symmetric QHD system and prove the wave function lifting method for spherically symmetric data, and as a consequence the global existence Theorem \ref{thm:glob_s}, \ref{thm:glob2_s} for spherically symmetric initial data is also obtained in Section \ref{sect:lift_s}. In Section \ref{sect:apri_md} we provide some a priori estimates on solutions to system \eqref{eq:QHD_md}, and finally we show the stability of solutions in Section \ref{sect:comp_md}.
	
	\section{Notations and Preliminaries}\label{sect:prel}

	In this Section we fix the notation and provide some preliminary results that will be used through this paper.
	
	The standard notation for Lebesgue and Sobolev norms are given by
	
	\[
	||f||_{L_{x}^{p}}\coloneqq(\int_{\R^d}|f(x)|^{p}dx)^{\frac{1}{p}},
	\]
	\[
	||f||_{W_{x}^{k,p}}\coloneqq\sum_{j=0}^k||\nabla^j f||_{L_{x}^{p}},
	\]
	and we denote $H_{x}^{k}\coloneqq H^{k}(\R^d)=W^{k,2}(\R^d)$.
	The mixed Lebesgue norm of functions $f:I\to L^{r}(\R^d)$ is
	defined as 
	\[
	||f||_{L_{t}^{q}L_{x}^{r}}\coloneqq\left(\int_{I}||f(t)||_{L_{x}^{r}}^{q}dt\right)^{\frac{1}{q}}=\left(\int_{I}(\int_{\R^d}|f(x)|^{r}dx)^{\frac{q}{r}}dt\right)^{\frac{1}{q}},
	\]
	where $I\subset[0,\infty)$ is a time interval. Similarly the mixed
	Sobolev norm $L_{t}^{q}W_{x}^{k,r}$ is defined.

	We recall here some basic properties of the following nonlinear Schr\"odinger equation
	\begin{equation}\label{eq:NLS_md}
		\left\{\begin{aligned}
			i\d_t\psi=&-\frac12\triangle\psi+|\psi|^{2(\gamma-1)}\psi\\
			\psi(0)=&\psi_0,
		\end{aligned}\right.
	\end{equation}
	and we mainly consider $\gamma\in(1, \infty)$ for $x\in\R^2$ and $1<\gamma<3$ for $x\in\R^3$.
	
	The reader can find more details and the proofs of the statements of the following theorem in the comprehensive monographs \cite{Caz, LP, Tao}.
	\begin{thm}\label{thm:NLS}
		Let $\psi_0\in H^1(\R^d)$, $\gamma>1$ for $d=2$ and $1<\gamma<3$ for $d=3$, then there exists a unique global solution $\psi\in\mathcal C(\R; H^1(\R^d))$ to \eqref{eq:NLS_md} such that the total mass and energy are conserved at all times, i.e.
		\begin{equation}\label{eq:cons_NLS}
			\begin{aligned}
				M(t)=&\int_{\R^d}|\psi(t, x)|^2\,dx=M(0)\\
				E(t)=&\int_{\R^d}\frac{1}{2}|\nabla\psi|^{2}+\frac1\gamma|\psi|^{2\gamma}\,dx=E(0).
		\end{aligned}\end{equation}
		If moreover $\psi_0\in H^2(\R^d)$, then we also have $\psi\in\mathcal C(\R;H^2(\R^d))\cap\mathcal C^1(\R;L^2(\R^d))$ and for any $0<T<\infty$ 
		\begin{equation}\label{eq:H2}
			\|\psi\|_{L^\infty(0, T;H^2(\R^d))}+\|\d_t\psi\|_{L^\infty(0, T; L^2(\R^d))}\leq C(T, \|\psi_0\|_{H^2(\R^d)}).
		\end{equation}
	\end{thm}
	
	Next we are going to recall the polar factorization technique developed in \cite{AM1, AM2}, see also the reviews \cite{AMDCDS, AM3}. This method allows to define the hydrodynamic states 
	$(\sqrt{\rho}, \Lambda)$ and sets up a correspondence between the wave function dynamics and the hydrodynamical system.
	The main advantage of this approach with respect to the usual WKB method for instance is that vacuum regions are allowed in the theory. Here we only give a brief introduction and state the results we will exploit later, for a more detailed presentation and for a proof of the statements in Lemma \ref{lemma:polar} below we address to Section 3 in \cite{AHMZ}.
	
	Given any function $\psi\in H^1(\R^d)$, we can define the set of polar factors as
	\begin{equation}\label{eq:set_pol}
		P(\psi)\coloneqq\left\{ \phi\in L^\infty(\R^d)\ |\ \|\phi\|_{L^{\infty}}\leq1,\ \sqrt{\rho}\phi=\psi\ a.e.\right\} ,
	\end{equation}
	where $\sqrt{\rho}:=|\psi|$. 
	In general this set can contain more than one element, due to the possible appearance of vacuum regions. Nevertheless, as it will be shown in the next lemma, the hydrodynamical states $(\sqrt{\rho}, \Lambda)$ are well defined and they furthermore enjoy suitable stability properties.
	
	\begin{lem}[Polar factorization \cite{AM1, AM2}]\label{lemma:polar}
		Let $\psi\in H^{1}(\R^d)$ and $\sqrt{\rho}:=|\psi|$, and let $\phi\in P(\psi)$. Then we have $\nabla\sqrt{\rho}=\RE(\bar\phi\nabla\psi)\in L^2(\R^d)$ and if we set 
		$\Lambda:= Im(\overline{\phi}\nabla\psi)$, we have 
		\begin{equation}\label{eq:quad}
			\RE(\nabla\overline{\psi}\otimes\nabla\psi)= \nabla\sqrt{\rho}\otimes\nabla\sqrt{\rho}
			+\Lambda\otimes\Lambda,\quad{\rm a.e.\; in}\;\R^d.
		\end{equation}
		Moreover if $\{\psi_n\}\subset H^1$ satisfies $\|\psi_n-\psi\|_{H^1}\to0$ as $n\to\infty$, then we have
		\begin{equation*}
			\nabla\sqrt{\rho_n}\to\nabla\sqrt{\rho},\quad\Lambda_n\to\Lambda,\quad{\rm in}\;L^2,
		\end{equation*}
		with $\sqrt{\rho_n}:=|\psi_n|, \Lambda_n:= \IM(\bar\phi_n\nabla\psi_n)$, $\phi_n$ being a unitary factor for $\psi_n$.
		\newline
		Finally, by defining $J=\sqrt{\rho}\Lambda$, then the following identity is satisfied in the sense of distributions
		\begin{equation}\label{eq:gic_pf}
			\nabla\wedge J=2\nabla\sqrt{\rho}\wedge\Lambda.
		\end{equation}
	\end{lem}
	
	We stress here the fact that $\Lambda=\IM(\bar\phi\nabla\psi)$ is well-defined even if $\phi$ is not uniquely determined, this is a consequence of Theorem 6.19 in \cite{LL}, which we will state below as it will be thoroughly used in our exposition.
	\begin{lem}\label{lemma:LL}
		Let $g:\Omega\to\mathbb{R}$ be in $H^{1}(\Omega)$, and 
		\[
		B=g^{-1}(\{0\})=\left\{ x\in\Omega\ :\ g(x)=0\right\} .
		\]
		
		Then $\nabla g(x)=0$ for almost every $x\in B$.
	\end{lem}
	
	By combining the well-posedness result for the NLS equation \eqref{eq:NLS_md} stated in Theorem \ref{thm:NLS} and the polar factorization method recalled in Lemma \ref{lemma:polar} 
	we can prove an existence result for finite energy weak solutions to \eqref{eq:QHD_md}, see Theorem \ref{thm:QHD_old} below. This result is  proved in \cite{AM1, AM2} for the two and three dimensional cases, see also the review \cite{AHMZ}.
	
	Before stating the results let us first recall a useful identity regarding the nonlinear dispersive term on the right hand side of the equation for the momentum density in \eqref{eq:QHD_md}, namely it can also be written as
	\begin{equation}\label{eq:bohm_md}
		\frac{1}{2} \rho\nabla\left(\frac{\triangle\sqrt{\rho}}{\sqrt{\rho}}\right)
		=\frac{1}{4}\diver(\rho\nabla^2\log\rho)
		=\frac{1}{4}\nabla\triangle\rho-\diver(\nabla\sqrt{\rho}\otimes\nabla\sqrt{\rho}).
	\end{equation}
	By using identity \eqref{eq:bohm_md}, system \eqref{eq:QHD_md} then becomes
	\begin{equation}\label{eq:QHD_md_2}
		\left\{\begin{aligned}
			&\d_t\rho+\diver J=0\\
			&\d_tJ+\diver\left(\Lambda\otimes\Lambda+\nabla\sqrt{\rho}\otimes\nabla\sqrt{\rho}\right)+\nabla p(\rho)=\frac14\nabla\triangle\rho
		\end{aligned}\right.
	\end{equation}
	We can now give the definition of weak solutions to the QHD system.
	
	\begin{defn}[Weak solutions]\label{def:FEWS}
		Let $\rho_0, J_0\in L^1_{loc}(\R^d)$, we say the pair $(\rho, J)$ is a finite energy weak solution to the Cauchy problem for \eqref{eq:QHD_md} with initial data
		$
		\rho(0)=\rho_0$, $J(0)=J_0,
		$
		in the space-time slab $[0, T)\times\R^d$ if there exist  
		$
		\sqrt{\rho}\in L^2_{loc}(0, T;H^1_{loc}(\R^d))$, $\Lambda\in L^2_{loc}(0, T;L^2_{loc}(\R^d))
		$
		such that
		\begin{itemize}
			\item[(i)] $\rho:=(\sqrt{\rho})^2, J:=\sqrt{\rho}\Lambda$;
			\item[(ii)] $\forall\;\eta\in\mathcal C^\infty_0([0, T)\times\R^d)$,
			\begin{equation}\label{eq:QHD_cty}
				\int_0^T\int_{\R^d}\rho\d_t\eta+J\cdot\nabla\eta\,dxdt+\int_{\R^d}\rho_0(x)\eta(0, x)\,dx=0;
			\end{equation}
			\item[(iii)] $\forall\zeta\in\mathcal C^\infty_0([0, T)\times\R^d;\R^d)$,
			\begin{equation}\label{eq:QHD_mom}
				\begin{aligned}
					\int_0^T\int_{\R^d}&J\cdot\d_t\zeta+\Lambda\otimes\Lambda:\nabla\zeta+p(\rho)\diver\zeta+\nabla\sqrt{\rho}\otimes\nabla\sqrt{\rho}:\nabla\zeta\\
					&+\frac{1}{4}\rho\Delta\diver\zeta\,dxdt
					+\int_{\R^d}J_0(x)\cdot\zeta(0, x)\,dx=0;
				\end{aligned}
			\end{equation}
			\item[(iv)] (generalized irrotationality condition) for almost every $t\in(0, T)$
			\begin{equation}\label{eq:gic_ws}
				\nabla\wedge J=2\nabla\sqrt{\rho}\wedge\Lambda,
			\end{equation}
			holds in the sense of distributions.
		\end{itemize}
		We say $(\rho, J)$ is a global in time weak solution if we can take $T=\infty$ in the above definition.
	\end{defn}
	
	\begin{rem}\label{rmk:vort}
		In the case of a smooth solution $(\rho, J)$, for which we can write $J=\rho v$, for some smooth velocity field $v$, the generalized irrotationality condition defined above is equivalent to $\rho\nabla\wedge v=0$, i.e. the velocity field $v$ is irrotational $\rho\,dx$ almost everywhere. It shows that the previous definition is the right weak formulation of the classical irrotationality condition $\nabla\wedge v=0$ valid away from vacuum in the WKB approach. The generalized irrotationality condition is also motivated by the quantised vortices of quantum flows, namely the flow is irrotational outside the set $\{\rho=0\}$ and in the vacuum the vorticity becomes singular. In this respect the solutions introduced in Definition \ref{def:FEWS} are more general than those obtained by using the WKB ansatz, since in the latter case the velocity field $v=\nabla S$ is always irrotational and there is no vacuum. Moreover, as we will explain in Remark \ref{rmk:qv_gic} below, it is possible to prove that hydrodynamic states satisfying the quantized vorticity condition, see \eqref{eq:QV_thm} for instance, also satisfy the generalized irrotationality condition. 
	\end{rem}
	
	In our paper we will always focus on weak solutions with finite mass and energy. 
	\begin{defn}[Finite mass and finite energy weak solutions]\label{def:FEWS_2}
		Let $(\rho,J)$ be a weak solution to the QHD system \eqref{eq:QHD_md} as in the Definition \ref{def:FEWS}. Let $M(t)$ be the total mass of $(\rho,J)$,
		\begin{equation*}
			M(t)=\int\rho(t, x)\,dx,
		\end{equation*}
		and $E(t)$ be the total energy denoted by
		\begin{equation*}
			E(t)=\int\frac12|\Lambda|^2+\frac12|\nabla\sqrt{\rho}|^2+f(\rho)\,dx.
		\end{equation*}
		We say $(\rho,J)$ is a finite mass weak solution, if for almost every $t\in[0, T)$ we have $M(t)\leq M(0)$. Analogously, $(\rho,J)$ is called a finite energy weak solution, if for almost every $t\in[0, T)$, we have $E(t)\leq E(0)$.
	\end{defn}
	
	\begin{rem}
		For a weak solution $(\rho, J)$ with finite initial mass, i.e. $M(0)<\infty$, it is straightforward to see that the total mass $M(t)$ is conserved for almost every time, see for example Theorem 1.3.4 in \cite{Daf}. A similar argument can also be applied to the total momentum
		\begin{equation*}
			P(t)=\int J(t, x)\,dx,
		\end{equation*}
		which is conserved for weak solutions with finite mass and finite energy.
	\end{rem}
	Without causing any confusion in this paper, we will always write finite energy weak solutions to mean weak solutions with both finite energy and mass.
	
	For initial data generated by wave functions, the existence of finite energy weak solutions was proved in \cite{AM1, AM2} (see also Theorem 4.2. in \cite{AHMZ}).
	\begin{thm}\label{thm:QHD_old}
		Let $\psi_0\in H^1(\R^d)$ and define the initial data for the QHD system \eqref{eq:QHD_md} as $\rho_0:=|\psi_0|^2, J_0:= \IM(\bar\psi_0\nabla\psi_0)$. Then there exists a global in time finite energy weak solution such that $\sqrt{\rho}\in L^\infty(\R; H^1(\R^d))$, 
		$\Lambda\in L^\infty(\R; L^2(\R^d))$, which conserves the energy at all times.
	\end{thm}
	
	\begin{rem}\label{rmk:vdW}
		Let us remark that the proof of Theorem \ref{thm:QHD_old} relies on two facts, the polar factorization technique of Lemma \ref{lemma:polar} and the well-posedness result in the energy space for the associated NLS equation. This means that in general the result in Theorem \ref{thm:QHD_old} remains true as long as the associated NLS equation is globally well-posed and enjoys suitable stability properties. In particular this implies that it is also possible to study QHD systems with non-monotone pressure terms.
	\end{rem}
	
	We conclude this section by giving a discussion about the higher order functional introduced in \eqref{eq:higher_md} and a definition for the generalized chemical potential $\lambda$ involved, formally related to the chemical potential $\mu$ by $\lambda=\sqrt{\rho}\mu$.
	
	As mentioned in the Introduction, to obtain the stability property of weak solutions, we will exploit the a priori bounds inferred from the functional 
	\[
	I(t)=\int_{\R^d} \lambda^2+(\d_t\sqrt\rho)^2dx,
	\]
	where $\lambda\in L^\infty(0,T;L^2(\R^d))$ is implicitly given by
	\begin{equation}\label{eq:xi_def}
		\xi=\sqrt\rho\lambda=-\frac14\triangle\rho+\frac12|\nabla\sqrt\rho|^2+\frac12|\Lambda|^2+f'(\rho)\rho.
	\end{equation}
	The functional $I(t)$ has a direct interpretation for Schr\"odinger-generated solutions. Indeed let us assume for the moment that we have $\rho=|\psi|^2$ and $J=\IM(\bar\psi\nabla\psi)$, for some solutions $\psi$ to \eqref{eq:NLS_md}. Then we can see that the functional \eqref{eq:higher_md} actually equals
	\begin{equation*}
		I(t)=\int|\d_t\psi|^2\,dx.
	\end{equation*}
	Indeed by using the polar factorization we have 
	\[
	\d_t\sqrt\rho=\RE(\bar\phi\d_t\psi),\quad \lambda=-\IM(\bar\phi\d_t\psi),
	\]
	where $\phi\in P(\psi)$ is a polar factor. On the other hand, by invoking Lemma \ref{lemma:LL}, we see that for Schr\"odinger-generated functions, $\lambda$ vanishes almost everywhere on the vacuum region $\{\rho=0\}$. 
	Let us remark that in this case identity \eqref{eq:xi_def} is recovered from $\xi=\sqrt{\rho}\lambda=-\IM(\bar\psi\d_t\psi)$ and by  exploiting the fact that $\psi$ is a solution to \eqref{eq:NLS_md}.
	Following this intuition, for arbitrary hydrodynamic states, not necessarily Schr\"odinger-generated, the rigorous definition of $\lambda$ is given as the following.
	
	\begin{defn}\label{def:lambda}
		Let $(\sqrt\rho,\Lambda)\in H^1(\R^d)\times L^2(\R^d)$ be a hydrodynamic state such that $\rho$ is continuous, $\triangle\rho\in L^1_{loc}(\R^d)$ and $\Lambda=0$ a.e. on $\{\rho=0\}$. Then we define the generalized chemical potential (GCP) $\lambda$ to be measurable function given by
		\begin{equation}\label{eq:def_lambda}
			\lambda=\left\{\begin{array}{cc}
				-\frac12\triangle\sqrt{\rho}+\frac12\frac{|\Lambda|^2}{\sqrt{\rho}}+f'(\rho)\sqrt{\rho}&\textrm{in }\;\{\rho>0\}\\
				0&\textrm{elsewhere}
			\end{array}\right.
		\end{equation}
	\end{defn}
	
	Note however that in order to prove our stability result given in Theorem \ref{thm:stab_md} we will only need $\lambda$ to satisfy
	$\lambda\in L^\infty_tL^2_x$ and identity \eqref{eq:xi_def}.
	In fact we will identify the class of weak solutions with bounded generalized chemical potential, for which we can prove our stability result exactly by those two properties satisfied by $\lambda$.
	
	\begin{defn}\label{def:cptsln}
		Let $(\rho, J)$ be a finite energy weak solution to \eqref{eq:QHD_md} on 
		$[0,T)\times\R^d$. We say that the $(\rho, J)$ is a GCP solution for the system \eqref{eq:QHD_md} if and only if
		\begin{itemize}
			\item there exists $\lambda\in L^\infty(0, T;L^2(\R^d))$, such that
			\begin{equation*}
				\sqrt{\rho}\lambda=-\frac14\triangle\rho+e+p(\rho),
			\end{equation*}
			where the energy density $e$ is defined by
			\begin{equation*}
				e=\frac12|\nabla\sqrt{\rho}|^2+\frac12|\Lambda|^2+f(\rho);
			\end{equation*}
			\item the following uniform bounds are satisfied
			\begin{equation*}
				\begin{aligned}
					\|\sqrt{\rho}\|_{L^\infty(0, T; H^1(\R^d))}+\|\Lambda\|_{L^\infty(0, T; L^2(\R^d))}\leq& M_1,\\
					\|\d_t\sqrt{\rho}\|_{L^\infty(0, T; L^2(\R^d))}+\|\lambda\|_{L^\infty(0, T; L^2(\R^d))}\leq& M_2.
				\end{aligned}
			\end{equation*}
		\end{itemize}
	\end{defn}
	
	\section{Wave function lifting: hydrodynamic data with points vacuum}\label{sect:lift_2d}
	
	The main tool we use to prove the global existence results is the wave function lifting method. For given initial data with hydrodynamic state $(\sqrt\rho_0,\Lambda_0)$ with suitable assumptions, we show that it is possible to define a wave function $\psi_0$ associated to $(\sqrt\rho_0,\Lambda_0)$ in the sense of the Definition \ref{def:ass} below. The global existence of weak solutions to $\eqref{eq:QHD_md}$ is obtained by solving the NLS equation with initial data $\psi_0$ and then using the polar factorization and Theorem \ref{thm:QHD_old}.
	
	\begin{defn}\label{def:ass}
		Let $\Omega\subset\R^d$ be an arbitrary open set. Given the hydrodynamic state $(\sqrt\rho,\Lambda)\in H^1(\Omega)\times L^2(\Omega)$, we say that the wave function $\psi\in H^1(\Omega)$ is associated to $(\sqrt{\rho}, \Lambda)$ if
		\begin{equation}\label{eq:polar}
			\sqrt\rho=|\psi|\quad \textrm{and} \quad \Lambda=\IM(\bar\phi\nabla\psi),
		\end{equation}
		where $\phi\in P(\psi)$ is a polar factor of $\psi$. 
	\end{defn}
	
	This section mostly focuses on the two-dimensional case. We will assume that the mass density $\rho = \rho(x)$, $x\in\R^2$, is continuous and that the vacuum set $V=\{x;\rho(x)=0\}$ consists of isolated points, namely there exists an at most countable set $\mathcal{A}$ of indices, such that
	\begin{equation}
		V=\{x_{(\alpha)}\}_{\alpha\in\mathcal{A}}\subset\R^2,\quad \inf_{\alpha\ne \beta}|x_{(\alpha)}-x_{(\beta)}|>0.
	\end{equation}
	Moreover we also consider a quantization condition for the vorticity. More precisely, let $v=J/\rho$ be the velocity field, which is well-defined almost everywhere, then the quantised vorticity condition is stated as following: 
	\begin{equation}\label{eq:QV_3.3}
		\begin{cases}
			v\in \mathcal M(\R^2)\\
			\nabla\wedge v=2\pi \underset{\alpha\in\mathcal{A}}{\sum}k_\alpha\delta_{x_{(\alpha)}}, \hspace{0.3cm} k_\alpha\in\Z,
		\end{cases}
	\end{equation}
	where $D'(\R^2)$ is the space of distribution, and $\delta_{x_{(\alpha)}}$ is the Dirac-delta function supported at $x_{(\alpha)}$ for $\alpha\in\mathcal{A}$. This condition is connected to the Bohr-Sommerfeld quantization rule in the old quantum theory.
	A straightforward consequence of the wave function lifting established in Proposition \ref{prop:lift_2d} below is that the quantised vorticity condition \eqref{eq:QV_3.3} implies the generalized irrotationality condition
	\begin{equation}\label{eq:gic_wfl}
		\nabla\wedge J=2\nabla\sqrt{\rho}\wedge \Lambda, \quad\textrm{for a.e.}\;x\in\R^2.
	\end{equation}
	see Remark \ref{rmk:qv_gic} for more details.

	We always assume $(\rho, J) = ((\sqrt{\rho})^2, \sqrt{\rho}\Lambda)
	$ has finite mass $M(\rho)$ and finite energy $E(\rho,J)$, which is equivalent to 
	\begin{equation*}
		\|\sqrt\rho\|_{H^1(\R^2)}+\|\Lambda\|_{L^2(\R^2)}\le M_1.
	\end{equation*}
	On the other hand, in the case of pointwise vacuum to characterise the higher order regularity of the state, we assume the condition that $\triangle\rho,\ \nabla J\in L^1_{loc}(\R^2)$ and 
	\begin{equation*}
		\|\frac{\Lambda_j^2}{\sqrt{\rho}}-\d_{x_j}^2\sqrt{\rho}\|_{L^2(\R^2)}
		+\|\frac{\d_{x_j} J_j}{\sqrt{\rho}}\|_{L^2(\R^2)}\leq M_2,\quad j=1,2,
	\end{equation*}
	where the functions in the right hand side of the inequality are well-defined a.e. on $\R^2$.
	
	In what follows the notion of difference quotients (see \cite{Evans}), namely for any function $f\in L^1_{loc}(\R^2)$ in direction $h\in \R^2$, $|h|>0$, we define
	\begin{equation}\label{eq:dq}
		D^h f(x)=\frac{f(x+h)-f(x)}{|h|}.
	\end{equation}
	
	Now we are at the point to prove the wave function lifting proposition. We remark that the continuity of density $\rho$, which can not be inferred by the finite energy assumption and Sobolev embedding in multi-dimension, is essential to give a proper definition to the point vacuum, and it is also needed to provide a local lower bound on the density when strictly away from the vacuum. Nevertheless, in the physical and mathematical literature the most widely studied vortex structures fall in this framework, see for example \cite{BJS,OS,RG}.
	
	\begin{prop}\label{prop:lift_2d}
		Let $(\sqrt{\rho}, \Lambda)$ be a pair of finite energy hydrodynamic state satisfying the bounds 
		\begin{equation}\label{eq:C1-2}
			\|\sqrt\rho\|_{H^1(\R^2)}+\|\Lambda\|_{L^2(\R^2)}\le M_1.
		\end{equation}
		Let us further assume that $\sqrt{\rho}$ is continuous with isolated point vacuum 
		\begin{equation}\label{eq:vac}
			V=\{x_{(\alpha)}\}_{\alpha\in\mathcal{A}}\subset\R^2,\quad \inf_{\alpha\ne \beta}|x_{(\alpha)}-x_{(\beta)}|>0,
		\end{equation}
		and the velocity $v$ satisfies the quantised vorticity conditions
		\begin{equation}\label{eq:QV}
			\begin{cases}
				v\in \mathcal M(\R^2)\\
				\nabla\wedge v=2\pi \underset{\alpha\in\mathcal{A}}{\sum}k_\alpha\delta_{x_{(\alpha)}}, \hspace{0.3cm} k_\alpha\in\Z.
			\end{cases}
		\end{equation}
		Then there exists a wave function 
		$\psi\in H^1(\R^2)$ such that
		\begin{equation*}
			\sqrt{\rho}=|\psi|, \quad\Lambda=\IM(\bar\phi\nabla\psi),
		\end{equation*}
		where $\phi\in P(\psi)$ is a polar factor defined in \eqref{eq:set_pol}.
		
		If we furthermore assume that $\triangle\rho,\nabla J\in L^1_{loc}(\R^2)$ and $(\sqrt{\rho}, \Lambda)$ satisfy also the bounds 
		\begin{equation}\label{eq:C2-2}
			\|\frac{\Lambda_j^2}{\sqrt{\rho}}-\d_{x_j}^2\sqrt{\rho}\|_{L^2(\R^2)}
			+\|\frac{\d_{x_j}J_j}{\sqrt{\rho}}\|_{L^2(\R^2)}\leq M_2,\quad j=1,2,
		\end{equation}
		then $\psi\in H^2(\R^2)$ and we have
		\begin{equation*}
			\|\psi\|_{H^2(\R^2)}\leq C(M_1, M_2).
		\end{equation*}
	\end{prop}
	
	\begin{rem}\label{rmk:qv_gic}
		A straightforward consequence of the wave function lifting stated in the above Proposition is that the quantized vorticity condition \eqref{eq:QV_3.3} implies the generalized irrotationality condition \eqref{eq:gic_wfl}. More precisely, given a hydrodynamic state satisfying the assumptions of Proposition \ref{prop:lift_2d}, then we know there exists an associated wave function $\psi\in H^1(\R^2)$. On the other hand, from the polar factorization Lemma \ref{lemma:polar} we further have that the hydrodynamic state $(\sqrt{\rho}, \Lambda)$ satisfy the identity \eqref{eq:gic_pf}, where $J=\sqrt{\rho}\Lambda$.
	\end{rem}
	
	\begin{proof}
		We start with the simplest case, the vacuum is a single point, and without loss of generality we assume $V=\{0\}$. In this case the quantised vorticity condition becomes 
		\begin{equation}\label{eq:QV_0}
			\begin{cases}
				v\in \mathcal M(\R^2)\\
				\nabla\wedge v=2k\pi \delta_{0}, \hspace{0.3cm} k\in\Z.
			\end{cases}
		\end{equation}
		For isolated vacuum points, the proof follows essentially same idea.
		
		To construct a wave function we take a sequence of smooth mollifier $\{\chi_\epsilon\}$ such that $supp(\chi_\epsilon)\subset B_\epsilon(0)$, where $B_r(x)$ is the ball centred at $x$ with radius $r$. We define
		\[
		v^\epsilon(x)=v\ast \chi_\epsilon (x)=\langle v(\cdot),\chi_\epsilon(x-\cdot)\rangle,
		\]
		so that
		\begin{align*}
			\nabla\wedge v^\epsilon=(\nabla \wedge v)\ast \chi_\epsilon=2k\pi \delta_0\ast\chi_\epsilon=2k\pi \chi_\epsilon.
		\end{align*}
		Let us fix an arbitrary $x_0\notin B_\epsilon(0)$, we can take a piecewise smooth curve $\gamma(x)$ connecting $x_0$ and $x$, strictly away from $B_\epsilon(0)$. In this way, for any $x\notin B_\epsilon(0)$, the approximating polar factor $\phi^\epsilon$ can be defined as
		\begin{equation*}
			\phi^\epsilon(x)=\exp(i\int_{\gamma(x)}v^\epsilon(y)\cdot d\vec{l}(y))
		\end{equation*}
		and we extend the definition of $\phi^\epsilon$ continuously to $x\in B_\epsilon(0)$ by setting $\phi^\epsilon(x)=\phi^\epsilon(\epsilon x/|x|)$. 
		\newline
		We claim that the above definition of $\phi^\epsilon$ is independent of the choice of curve $\gamma$. Indeed, for any two curves $\gamma_1$ and $\gamma_2$ connecting $x_0$ and $x$, we can patch them into a piecewise smooth closed curve $\tilde{\gamma}$. Denote $\Omega$ the domain inside $\tilde{\gamma}$, then we have two cases: $B_\epsilon(0)\subset\Omega$ or $B_\epsilon(0)\cap\Omega=\O$. In either case by Stokes' theorem 
		\begin{align*}
			\int_{\gamma_1}v^\epsilon(y)\cdot d\vec{l}(y)-\int_{\gamma_2}v^\epsilon(y)\cdot d\vec{l}(y)&=\int_{\tilde{\gamma}}v^\epsilon(y)\cdot d\vec{l}(y)\\
			&=\int_\Omega \nabla\wedge v^\epsilon(y)dy\\
			&=2k\pi\ or\ 0,
		\end{align*}
		where $k$ is the integer given in \eqref{eq:QV_0}, thus we have 
		\[
		\exp(i\int_{\gamma_1}v^\epsilon(y)\cdot d\vec{l}(y))=\exp(i\int_{\gamma_2}v^\epsilon(y)\cdot d\vec{l}(y)).
		\]
		\newline
		The polar factors $\{\phi^\epsilon\}$ have uniform $L^\infty$ bound $1$, therefore up to subsequences we have $\phi^{\epsilon}\overset{\ast}{\rightharpoonup}\phi$ in $L^\infty(\R^2)$ with $\|\phi\|_{L^\infty}\leq 1$. The lifted wave function is defined by $\psi= \sqrt{\rho}\phi$. Obviously $\psi\in L^2(\R^2)$, then we are going to show $\psi\in H^1(\R^2)$.
		
		For this purpose, we only need to show the $L^2$ norm of the difference quotient $D^h\phi$ defined in \eqref{eq:dq} has a uniform bound, namely
		\[
		\|D^h\psi\|_{L^2(\R^2)}\le C
		\]
		for some constant $C$ and all $h\in \R^2$ with $|h|\ne 0$ small. We have
		\[
		\|D^h\psi\|_{L^2(\R^2)}=\|D^h\psi\|_{L^2(B_{|h|}(0))}+\|D^h\psi\|_{L^2(B_{|h|}(0)^c)}.
		\]
		For the part near the vacuum point $0$, we simply bound it by 
		\begin{align*}
			\|D^h\psi\|_{L^2(B_{|h|}(0))}&=\frac{1}{|h|}\|\psi(\cdot+h)-\psi(\cdot)\|_{L^2(B_{|h|}(0))}\\
			&\le \frac{2}{|h|}\|\sqrt{\rho}\|_{L^\infty(B_1(0))}|B_{|h|}(0)|^{\frac{1}{2}}=2\sqrt{\pi}\|\sqrt{\rho}\|_{L^\infty(B_1(0))},
		\end{align*}
		which is uniformly bounded since $\sqrt{\rho}$ is continuous.
		
		Then we consider the term away from the vacuum. For any relatively compact open set $U\subset B_{|h|}(0)^c$ and for any $\epsilon<|h|$, we have explicit representation
		\begin{align*}
			\phi^\epsilon&=\exp(i\int_{\gamma}v^\epsilon(y)\cdot d\vec{l}(y)),\\
			\nabla\phi^\epsilon&=i v^\epsilon \phi^\epsilon.
		\end{align*}
		Since $\sqrt{\rho}$ is continuous and $U$ is strictly away from the vacuum, then $\underset{x\in U}{\inf}\{\sqrt{\rho}(x)\}=\alpha>0$ and the velocity field $v=\Lambda/\sqrt{\rho}\in L^2(U)$. Moreover we have $v^\epsilon\to v$ in $L^2(U)$, and $\phi=*-\underset{\eps\to0}{\lim}\phi^\eps$,  then it follows
		\begin{equation}\label{eq:gradphi}
			\nabla\phi^\epsilon=i v^\epsilon \phi^\epsilon \rightharpoonup iv\phi=\nabla\phi\hspace{0.3cm}in\ L^2(U).
		\end{equation}
		Thus we obtain that in $L^2(U)$
		\begin{align*}
			\nabla\psi=(\nabla\sqrt{\rho}+i\sqrt{\rho}v)\phi=(\nabla\sqrt{\rho}+i\Lambda)\phi
		\end{align*}
		with uniform $L^2$ bound 
		\[
		\|\nabla\psi\|_{L^2(U)}=\|\nabla\sqrt{\rho}\|_{L^2(U)}+\|\Lambda\|_{L^2(U)}\le M_1.
		\]
		By using a sequence of relatively compact sets $U$ invading $B^c$, we get
		\[
		\|\nabla\psi\|_{L^2(B_h(0)^c)}\le M_1,
		\]
		and we conclude $\psi\in H^1(\R^2)$. Furthermore, the above argument implies for a.e. $x\in R^2$ we have
		\begin{equation}\label{eq:gradpsi}
			\nabla\psi=(\nabla\sqrt{\rho}+i\Lambda)\phi,
		\end{equation}
		which proves the polar factorization \eqref{eq:polar}.
		
		Further let us assume the condition \eqref{eq:C2-2}, then we are able to show that $\psi\in H^2(\R^2)$. The proof uses difference quotients as in the previous case, but here it requires further small technicalities since $\nabla\psi\notin L^\infty$.
		
		Let us consider a relatively compact set $U$, strictly away from the vacuum point $0$. By \eqref{eq:gradphi} and \eqref{eq:gradpsi} it follows that for $j=1,2,$
		\begin{align*}
			[(\d_{x_j}^2\sqrt{\rho}-\frac{\Lambda^2_j}{\sqrt{\rho}})+i\frac{\d_{x_j} J_j}{\sqrt{\rho}}]\phi&=(\d_{x_j}^2\sqrt{\rho}+i\d_{x_j}\Lambda_j+i\d_{x_j}\sqrt{\rho}\, v_j-\Lambda_j\,v_j)\phi\\
			&=\d_{x_j}\left[(\d_{x_j}\sqrt\rho+\Lambda_j)\phi\right]=\d_{x_j}(\d_{x_j}\psi)
		\end{align*}
		in the sense of distribution. Therefore the condition \eqref{eq:C2-2} implies $\d_{x_j}^2\psi\in L^2(U)$ satisfying the uniform bound 
		\begin{equation}\label{eq:dxj2psi}
			\|\d_{x_j}^2\psi\|_{L^2(U)}\leq M_2.
		\end{equation}
		Now we again take the smooth mollifier $\{\chi_\epsilon\}$ such that $supp(\chi_\epsilon)\subset B_\eps(0)$, and we define the smooth approximating wave function $\psi^\eps=\psi\ast\chi_\eps$. It follows \eqref{eq:dxj2psi} that for any relatively compact open set $U$ such that $dist(U,0)>\eps$, we have
		\begin{equation}\label{eq:dxj2psieps}
			\|\d_{x_j}^2\psi^\eps\|_{L^2(U)}\leq\|\d_{x_j}^2\psi\|_{L^2(U^\eps)}\leq M_2,
		\end{equation}
		where $U^\eps$ is the $\eps-$neighbourhood of the set $U$. For $\partial_{x_j}\psi^\eps$, $j=1,2$, by using the polar factorization we  can decompose it as $\partial_{x_j}\psi^\eps=|\partial_{x_j}\psi^\eps|\phi_j^\eps$, where the polar factor $\phi_j^\eps=\partial_{x_j}\psi^\eps/|\partial_{x_j}\psi^\eps|$ when $|\partial_{x_j}\psi^\eps|\ne 0$. In this case direct computation shows 
		\begin{equation}\label{eq:gradphik}
			\d_{x_j}\phi_j^\eps=i\frac{\IM(\partial_{x_j}\overline{\psi^\eps}\d_{x_j}^2\psi^\eps)}{|\partial_{x_j}\psi^\eps|^2}\phi_j^\eps,
		\end{equation}
		and by the same argument as polar factorization we have 
		\begin{equation}\label{eq:hesspsi}
			|\d_{x_j}^2\psi^\eps|^2=\left(\d_{x_j}|\partial_{x_j}\psi^\eps|\right)^2+\left(\frac{\IM(\partial_{x_j}\overline{\psi^\eps}\d_{x_j}^2\psi^\eps)}{|\partial_{x_j}\psi^\eps|}\right)^2.
		\end{equation}
		
		To apply the theory of difference quotients, instead of $\d_{x_j}\psi^\eps$ and $\d_{x_j}\psi$, we will consider the $L^\infty$ cut-off functions $\{f_{j,n}^\eps\}$ and $f_{j,n}$ defined by
		\begin{align*}
			f_{j,n}^\eps=&\min\{|\partial_{x_j}\psi^\eps|,n\}\phi_j^\eps,\quad j=1,2,\\
			f_{j,n}=&\min\{|\partial_{x_j}\psi|,n\}\phi_j,\quad \phi_j\in P(\d_{x_j}\psi).
		\end{align*}
		By using the fact
		\[
		|f_{j,n}^\eps-f_{j,n}|\leq |\d_{x_j}\psi^\eps-\d_{x_j}\psi|,
		\]
		and $\d_{x_j}\psi^\eps\to \d_{x_j}\psi$ in $L^2(\R^2)$, it follows $f_{j,n}^\eps\to f_{j,n}$ in $L^2(\R^2)$ as $\eps\to 0$. 
		
		Now we consider the difference quotient of $f_{j,n}$ in the direction $e_j$, where 
		\[
		e_1=(1,0)\quad and \quad e_2=(0,1),
		\]
		and we let $h_j=|h|\,e_j$ for $|h|\ne 0$ small, and we have
		\[
		\|D^{h_j} f_{j,n}\|_{L^2(\R^2)}=\|D^{h_j} f_{j,n}\|_{L^2(B_{|h|}(0))}+\|D^{h_j} f_{j,n}\|_{L^2(B_{|h|}(0)^c)}.
		\]
		For the part near the vacuum, it can be simply controlled by 
		\[
		\|D^{h_j} f_{j,n}\|_{L^2(B_{|h|}(0))}\le 2\sqrt{\pi}\|f_{j,n}\|_{L^\infty}\le 2\sqrt{\pi}n.
		\]
		To deal with the part away from the vacuum, we again consider a relatively compact open sets $U\subset B_{|h|}(0)^c$. By definition $|f_{j,n}^\eps|\le |\partial_{x_j}\psi^\eps|$, and since $\d_{x_j}^2\psi^\eps\in L^2(U)$ for all $\eps<|h|$, the $x_j-$derivative of the truncated function $f_{j,n}^\eps$ also belongs to $L^2(U)$. Moreover by using \eqref{eq:gradphik} and \eqref{eq:hesspsi}, for $|\partial_j\psi^\eps|> n$ we have
		\begin{align*}
			|\d_{x_j} f_{j,n}^\eps|^2=|\d_{x_j} (n\phi_j^\eps)|^2=\left(n\frac{\IM(\partial_{x_j}\overline{\psi^\eps}\partial_{x_j}^2\psi^\eps)}{|\partial_{x_j}\psi^\eps|^2}\right)^2
			\leq  \left(\frac{\IM(\partial_{x_j}\overline{\psi^\eps}\partial_{x_j}^2\psi^\eps)}{|\partial_{x_j}\psi^\eps|}\right)^2\leq|\d_{x_j}^2 \psi^\eps|^2,
		\end{align*}
		and $\d_{x_j} f_{j,n}^\eps=\d_{x_j}^2\psi^\eps$ in the case $|\d_{x_j}\psi^\eps|\leq n$. Therefore by using \eqref{eq:dxj2psieps}, we have the uniform bound 
		\begin{equation*}
			\|\d_{x_j} f_{j,n}^\eps\|_{L^2(U)}\leq \|\d_{x_j}^2\psi^\eps\|_{L^2(U)}\leq M_2.
		\end{equation*}
		By passing to the limit as $\eps\to 0$, we obtain
		\begin{equation}\label{eq:dxjfjn}
			\|\d_{x_j} f_{j,n}\|_{L^2(U)}\leq M_2
		\end{equation}
		for any relatively compact open set $U$ strictly away from $0$, hence it also works on $B_{|h|}(0)^c$. Therefore the theory of difference quotient shows $\d_{x_j}f_{j,n}\in L^2(\R^2)$ and we have the bound 
		\begin{equation*}
			\|\d_{x_j} f_{j,n}\|_{L^2(\R^2)}\leq M_2+2\sqrt\pi n.
		\end{equation*}
		
		Notice that even though the bound of $\|\d_{x_j} f_{j,n}\|_{L^2}$ obtained above depends on $n$, as long as we have $\d_{x_j}f_{j,n}\in L^2(\R^2)$, a zero measure set has no effect on the $L^2$ norm of $\d_{x_j}f_{j,n}$. Thus by choosing a sequence of open set $U$ invading $\R^2-\{0\}$ in \eqref{eq:dxjfjn}, we obtain that
		\[
		\|\d_{x_j} f_{j,n}\|_{L^2(\R^2)}\le M_2,\quad j=1,2.
		\]
		On the other hand, it is straightforward to see that $f_{j,n}$ converges to $\d_{x_j}\psi$ strongly as $n\to \infty$ by construction. Therefore we conclude $\d_{x_j}^2\psi\in L^2(\R^2)$ for $j=1,2$, with the bound 
		\[
		\|\d_{x_j}^2\psi\|_{L^2(\R^2)}\le M_2.
		\]
		It implies $\triangle \psi\in L^2(\R^2)$ and we have
		\[
		\|\triangle \psi\|_{L^2(\R^2)} \le 2\,M_2,
		\]
		then by the standard argument we can conclude $\psi\in H^2(\R^2)$ satisfying the bound
		\[
		\|\psi\|_{H^2(\R^2)}\leq C(M_1,M_2).
		\]
	\end{proof}
	
	The next proposition is a simplified analogue of Proposition \ref{prop:lift_2d} which will be used later. We consider hydrodynamic state $(\sqrt\rho,\Lambda)$ such that $\rho$ is continuous with $\rho>0$, and let us assume the data is irrotational in the sense of 
	\begin{equation}\label{eq:irr}
		\nabla\wedge J=2\nabla\sqrt{\rho}\wedge \Lambda.
	\end{equation}
	By the Remark \ref{rmk:vort} we see that when the density $\rho$ is strictly positive, identity \eqref{eq:irr} is equivalent to the classical irrotationality condition $\nabla\wedge v=0$. For such hydrodynamic state with the regularity assumptions \eqref{eq:C1-2} and \eqref{eq:C2-2}, we can lift it to a wave function $\psi\in H^2(\R^d)$ following the same approach as in Proposition \ref{prop:lift_2d}. Moreover the same argument also applies to general multi-dimensional data. 
	
	\begin{prop}\label{prop:lift_pos}
		Let $(\sqrt{\rho}, \Lambda)$ be a finite energy hydrodynamic state on $\R^d$, such that $\sqrt{\rho}$ is continuous with $\rho>0$ and it satisfies the bounds 
		\begin{equation}\label{eq:C1-2_pos}
			\|\sqrt\rho\|_{H^1(\R^d)}+\|\Lambda\|_{L^2(\R^d)}\le M_1.
		\end{equation}
		and 
		\begin{equation}\label{eq:C2-2_pos}
			\|\frac{\Lambda^2}{\sqrt{\rho}}-\triangle\sqrt{\rho}\|_{L^2(\R^d)}
			+\|\frac{\diver J}{\sqrt{\rho}}\|_{L^2(\R^d)}\leq M_2,
		\end{equation}
		Let us further assume $(\rho,J)$ satisfies the generalized irrotationality condition \eqref{eq:irr}. Then there exists a wave function 
		$\psi\in H^2(\R^d)$, such that
		\begin{equation*}
			\sqrt{\rho}=|\psi|, \quad\Lambda=\IM(\bar\phi\nabla\psi),
		\end{equation*}
		where $\phi\in P(\psi)$ is a polar factor, and we have controls
		\begin{equation*}
			\|\psi\|_{H^2(\R^d)}\leq C(M_1,M_2).
		\end{equation*}
	\end{prop}
	
	\begin{proof}
		The proof is analogue to the proof to Proposition \ref{prop:lift_2d}, but the argument becomes simpler due to the absence of vacuum. 
		
		Since $\sqrt{\rho}$ is continuous in space and strictly positive, it has a positive lower bound on any compact set in $\R^d$. Therefore the velocity field $v=\Lambda/\sqrt{\rho}$ belongs to $L^2_{loc}(\R^d)$. Let $\{\chi_\epsilon\}$, $\epsilon>0$ be a sequence of smooth mollifier, we define the approximating velocity 
		\[
		v^\epsilon=v\ast \chi_\epsilon,
		\]
		and the approximating polar factor 
		\[
		\phi^\epsilon(x)=\exp(i\int_{\gamma(x)}v^\epsilon(y)\cdot d\vec{l}(y)),
		\]
		where $\gamma(x)$ is the straight line in $\R^d$ connecting $0$ and $x$. By using the Remark \ref{rmk:vort}, the identity \eqref{eq:irr} is equivalent to the classical irrotationality condition
		\[
		\nabla \wedge v=0,
		\]
		when then density $\rho$ is always positive. Hence the approximating velocity field $v^\epsilon$ is also irrotational, and the polar factor $\phi^\epsilon$ is well-defined. Moreover it is straightforward to compute that 
		\[
		\nabla\phi^\epsilon=iv^\epsilon\phi^\epsilon.
		\]
		The uniform $L^\infty$ bound of polar factors $\{\phi^\epsilon\}$ allows us to pass to a weak limit upto subsequences, namely $\phi^{\epsilon}\overset{\ast}{\rightharpoonup}\phi$ in $L^\infty(\R^2)$ with $\|\phi\|_{L^\infty}\leq 1$. Therefore by using $v^\epsilon\to v$ in $L^2_{loc}(\R^d)$ as $\epsilon\to 0$, it follows
		\begin{equation}\label{eq:dphi_md}
			\nabla\phi^\epsilon=i\,v^\epsilon\phi^\epsilon\rightharpoonup iv\phi=\nabla\phi\quad in\;L^2_{loc}(\R^d).
		\end{equation}
		
		If we define $\psi=\sqrt{\rho}\phi$, then we can directly check that $\psi$ is the wave function associated to $(\sqrt\rho,\Lambda)$ and $\psi\in H^2(\R^2)$. Clearly $|\psi|=\sqrt{\rho}$. By identity \eqref{eq:dphi_md} we can compute 
		\begin{equation}\label{eq:111}
			\nabla\psi=(\nabla\sqrt\rho+i\sqrt\rho v)\phi=(\nabla\sqrt\rho+i\Lambda)\phi,
		\end{equation}
		hence $\Lambda=\IM(\bar\phi\nabla\psi)$ and 
		\[
		\|\psi\|_{H^1(\R^2)}\le \|\sqrt{\rho}\|_{H^1(\R^2)}+\|\Lambda\|_{L^2(\R^2)}\le M_1.
		\]
		Moreover by \eqref{eq:dphi_md} and \eqref{eq:111},
		\begin{align*}
			\triangle\psi&=\diver\nabla\psi\\
			&=(\triangle\sqrt{\rho}+i\diver\Lambda+i\nabla\sqrt{\rho}\cdot v-\Lambda\cdot v)\phi\\
			&=\left[(\triangle\sqrt{\rho}-\frac{|\Lambda|^2}{\sqrt{\rho}})+i\frac{\diver J}{\sqrt{\rho}}\right]\phi.
		\end{align*}
		Using \eqref{eq:C1-2_pos} we obtain that 
		\[
		\|\triangle \psi\|_{L^2(\R^2)}\leq \|\frac{\Lambda^2}{\sqrt{\rho}}-\triangle\sqrt{\rho}\|_{L^2(\R^2)}
		+\|\frac{\diver J}{\sqrt{\rho}}\|_{L^2(\R^2)}\leq M_2.
		\]
	\end{proof}
	
	To conclude this section we prove Theorem \ref{thm:glob_md} and Theorem \ref{thm:glob2_md}, namely the global existence theorem for the Cauchy problem of the QHD system \eqref{eq:QHD_md} in 2D. We assume initial data to be satisfying the point vacuum and quantised vorticity condition. The proof is obtained by combining Proposition \ref{prop:lift_s}, the global well-posedness of NLS equations (Theorem \ref{thm:NLS}) and the polar factorization method (Theorem \ref{thm:QHD_old}).
	
	\begin{proof}[Proof of Theorem \ref{thm:glob_md} and Theorem \ref{thm:glob2_md}]
		Let us consider initial data $(\rho_0, J_0)$ satisfying \eqref{eq:C1_md_intro} and the conditions \eqref{eq:vac_intro} and \eqref{eq:QV_intro}, then by Proposition \ref{prop:lift_2d} there exists a wave function $\psi_0\in H^1$ associated to $(\sqrt{\rho_0}, \Lambda_0)$. By using Theorem \ref{thm:NLS} and Theorem \ref{thm:QHD_old}, we obtain a global in time finite energy weak solution $(\sqrt{\rho},\Lambda)$ to \eqref{eq:QHD_md}, which conserves the energy. 
		
		If $(\sqrt{\rho_0}, \Lambda_0)$ satisfies both \eqref{eq:C1_md_intro} and \eqref{eq:C2_md_intro}, then Proposition \ref{prop:lift_2d} shows the wave function $\psi_0$ belongs to $H^2(\R^2)$. 
		By Theorem \ref{thm:NLS} for the NLS, we also have $\psi\in\mathcal C([0,T];H^2(\R^2))\cap\mathcal C^1([0,T];L^2(\R^2))$ with
		\begin{equation}\label{eq:H2_md}
			\|\psi\|_{L^\infty(0, T;H^2(\R^2))}+\|\d_t\psi\|_{L^\infty(0, T; L^2(\R^2))}\leq C(T,M_1,M_2).
		\end{equation}
		Now, let us recall that $\d_t\sqrt{\rho}=\RE(\bar\phi\d_t\psi)$ and $\lambda=-\IM(\bar\phi\d_t\psi)$, with $\phi\in P(\psi)$, then by the polar factorization and \eqref{eq:H2_md} we infer
		\begin{equation*}
			\|\d_t\sqrt{\rho}\|_{L^\infty_tL^2_x}+\|\lambda\|_{L^\infty_tL^2_x}\leq C(T, M_1, M_2).
		\end{equation*}
		Moreover using the NLS equation \eqref{eq:NLS_md} we can directly compute that
		\begin{equation*}
			\sqrt{\rho}\lambda=-\IM(\bar\psi\d_t\psi)=-\frac14\triangle\rho+e+p(\rho)
		\end{equation*}
		so that \eqref{eq:rho_mu_s} holds.
		To conclude the proof of Theorem \ref{thm:glob2_md} it only remains to prove the bounds on the higher order derivatives. 
		Since $\triangle\rho=2\RE(\bar\psi\triangle\psi)+2|\nabla\psi|^2$, then by H\"older's inequality and Sobolev embedding we have
		\begin{equation*}
			\|\triangle\rho\|_{L^\infty(0, T;L^2(\R^2))}\lesssim\|\psi\|^2_{L^\infty(0, T;H^2(\R^2))}\leq C(T, M_1, M_2).
		\end{equation*}
		On the other hand, by the Madelung transformation
		\begin{equation*}
			\|\nabla J\|_{L^\infty_tL^2_x}=\|\nabla\IM(\bar\psi\nabla\psi)\|_{L^\infty_tL^2_x}\leq\|\psi\|_{L^\infty_tH^2_x}^2\leq C(T,M_1,M_2).
		\end{equation*}
		Finally, let us recall that by the polar factorization we have 
		\[
		e=\frac12|\nabla\sqrt{\rho}|^2+\frac12|\Lambda|^2+f(\rho)=\frac12|\nabla\psi|^2+f(|\psi|^2).
		\]
		It is well-known that 
		$|\nabla|\nabla\psi||\leq|\nabla^2\psi|$ a.e. $x\in\R^2$, hence
		\begin{equation*}
			\|\nabla\sqrt{e}\|_{L^\infty_tL^2_x}\leq\|\psi\|_{L^\infty_tH^2_x}\leq C(T, M_1, M_2).
		\end{equation*}
	\end{proof}
	
	\subsection{Extension to 3D hydrodynamic data: straight vortex lines}
	Here we extend the wave function lifting argument to a simple example of 3D hydrodynamic data, where we assume planar symmetry to the data and the vorticity to be concentrated in a finite number of vortex lines. This is a typical example of the vortex structure of quantum flows in 3D space \cite{P}. More precisely, we consider hydrodynamic data $(\rho(x),J(x))$, $x=(x_1,x_2,x_3)\in\R^3$ and $J\in\R^3$, of the following form: 
	\begin{equation}\label{eq:cyldata}
		\rho(x)=\rho_1(x_1,x_2)\rho_2(x_3)\quad and \quad J(x)=\left(J_1(x_1,x_2)\rho_2(x_3),0\right)=\sqrt{\rho}\Lambda,
	\end{equation}
	where $\rho,\rho_j\in\R$ and $J_1\in\R^2$. The data $(\rho,J)$ is essentially a product of a 2D data $(\rho_1,J_1)$ and a 1D data $(\rho_2,0)$. We assume $(\rho_1, J_1) = ((\sqrt{\rho_1})^2, \sqrt{\rho_1}\Lambda_1)$ has pointwise vacuum and quantised vorticity, namely we assume $\rho_1$ to be continuous with vacuum structure \eqref{eq:vac}, and the velocity field $v_1=J_1/\rho_1$ to satisfy the quantized vorticity condition \eqref{eq:QV}. The regularity of $(\rho_1,J_1)$ is characterised by the bounds \eqref{eq:C1-2} and \eqref{eq:C2-2}, and we assume $\sqrt{\rho_2}\in H^s(\R)$ with $s=1$ or $2$.
	
	With the previous assumptions, we can state the following proposition as an extension of the wave function lifting argument to 3D data.
	Moreover, by using the next proposition, we can extend the global existence results to the case of 3D initial data of the form (3.18), namely Theorem \ref{thm:glob_3d} and Theorem \ref{thm:glob2_3d}.
	
	\begin{prop}\label{prop:lift_3D}
		Let $(\rho, J)$ be a pair of hydrodynamic data as in \eqref{eq:cyldata}, with $(\rho_1,J_1)$ satisfying the condition \eqref{eq:vac} and \eqref{eq:QV}. Furthermore let us assume that $(\rho_1,J_1)$ satisfy the mass and energy bounds \eqref{eq:C1-2} and $\sqrt{\rho_2}\in H^1(\R)$. Then there exists a wave function 
		$\psi\in H^1(\R^3)$ such that
		\begin{equation*}
			\sqrt{\rho}=|\psi|, \quad\Lambda=\IM(\bar\phi\nabla\psi),
		\end{equation*}
		where $\phi\in P(\psi)$ is a polar factor defined in \eqref{eq:set_pol}.
		
		If we furthermore assume the bounds \eqref{eq:C2-2} on $(\rho_1,J_1)$ and $\sqrt{\rho_2}\in H^2(\R)$, then $\psi\in H^2(\R^3)$ and 
		\begin{equation*}
			\|\psi\|_{H^2(\R^3)}\leq C(M_1, M_2).
		\end{equation*}
	\end{prop}
	
	\begin{proof}
		By applying Proposition \ref{prop:lift_2d} to $(\rho_1,J_1)$, we obtain a lifted wave function $\psi_1\in H^1(\R^2)$ such that 
		\begin{equation*}
			\sqrt{\rho_1}=|\psi_1|, \quad\Lambda_1=\IM(\bar\phi_1\nabla\psi_1),
		\end{equation*}
		where $\phi_1$ is a polar factor of $\psi_1$. Then we claim 
		\begin{equation*}
			\psi(x)=\psi_1(x_1,x_2)\sqrt{\rho_2(x_3)}
		\end{equation*}
		is the wave function associated to $(\sqrt{\rho}, \Lambda)$. Since $\psi_1\in H^1(\R^2)$ and $\sqrt{\rho_2}\in H^1(\R)$ we have $\psi\in H^1(\R^3)$. Furthermore 
		\begin{equation*}
			|\psi|=|\psi_1|\sqrt{\rho_2}=\sqrt{\rho_1\rho_2}=\sqrt{\rho}.
		\end{equation*}
		Obviously $\phi(x)=\phi_1(x_1,x_2)$ is a polar factor of $\psi$, then for $j=1,2$ 
		\begin{equation*}
			\IM(\bar\phi\d_{x_j}\psi)=\IM(\bar\phi_1\d_{x_j}\psi_1)\sqrt{\rho_2}=\Lambda_{1,j}\sqrt{\rho_2}=\Lambda_j,
		\end{equation*}
		and
		\begin{equation*}
			\IM(\bar\phi\d_{x_3}\psi)=\IM(\bar\phi_1\psi_1)\d_{x_3}\sqrt{\rho_2}=0.
		\end{equation*}
		Hence we conclude 
		\begin{equation*}
			\IM(\bar\phi\nabla\psi)=\Lambda.
		\end{equation*}
		
		Furthermore, if we assume $(\rho_1,J_1)$ to satisfy \eqref{eq:C2-2}, which implies $\psi_1\in H^2(\R^2)$ by Proposition \ref{prop:lift_2d}, and assume $\sqrt{\rho_2}\in H^2(\R)$, then it follows $\psi=\psi_1\sqrt{\rho_2}\in H^2(\R^3)$ with 
		\begin{equation*}
			\|\psi\|_{H^2(\R^3)}\leq \|\psi_1\|_{H^2(\R^2)}\|\sqrt{\rho_2}\|_{H^2(\R^2)}\leq C(M_1,M_2).
		\end{equation*}
	\end{proof}
	
	\section{Wave function lifting: spherical symmetric data}\label{sect:lift_s}
	
	This section concerns the wave function lifting method for spherically symmetric data, which is essentially an analogue of the method we applied in \cite{AMZ} to one-dimensional data. Then we prove the global in time existence of spherically symmetric solutions to system \eqref{eq:QHD_md}.
	
	\subsection{Spherically symmetric QHD system}
	
	A pair of spherically symmetric initial data $(\rho_0, J_0) = ((\sqrt{\rho_0})^2, \sqrt{\rho_0}\Lambda_0)$ is given by a hydrodynamic state of the form
	\begin{equation}\label{eq:109}
		\sqrt\rho_0(x)=\sqrt\rho_0(r)\ and\ \Lambda_0(x)=\Lambda_0(r)\frac{x}{|x|}
	\end{equation}
	for $x\in\R^d$ with $r=|x|$, and $\Lambda_0(r)\in \R$.
	In this case the corresponding solutions to the QHD system \eqref{eq:QHD_md} also satisfy the spherical symmetry,
	\[
	\rho(t,x)=\rho(t,r)\ and\ J(t,x)=J(t,r)\frac{x}{|x|}=\sqrt\rho\Lambda\frac{x}{|x|}.
	\]
	In this paper we mainly focus on the case $d=2$ or $d=3$, but the argument of this section applies to general dimension $d$.
	
	By using standard calculus formulae and polar coordinates, the proposition below describes the spherically symmetric formulation for the QHD system.
	
	\begin{prop}
		Let $r>0$,for spherically symmetric solutions of the form 
		\begin{equation*}
			\rho(t,x)=\rho(t,r)\, and\, J(t,x)=J(t,r)\frac{x}{|x|}=\sqrt\rho\Lambda\frac{x}{|x|},
		\end{equation*}
		the QHD system \eqref{eq:QHD_md} is equivalent to 
		\begin{equation}\label{eq:sQHD}
			\left\{\begin{aligned}
				&\d_t\rho+(\d_r+\frac{d-1}{r})(\sqrt\rho\Lambda)=0\\
				&\d_tJ+(\d_r+\frac{d-1}{r})(\Lambda^2+(\d_r\sqrt{\rho})^2)+\d_r p(\rho)=
				\frac14\d_r(\d_r+\frac{d-1}{r})\d_r\rho.
			\end{aligned}\right.
		\end{equation}
		The total mass and energy are defined as 
		\begin{equation}\label{eq:mass_s}
			M(t)=C(d)\int_0^\infty r^{d-1}\,\rho\,dr,
		\end{equation} 
		and 
		\begin{equation}\label{eq:en_s}
			E(t)=C(d)\int_0^r r^{d-1}[\frac12(\d_r\sqrt{\rho})^2+\frac12\Lambda^2+f(\rho)]dr,
		\end{equation}
		where $C(d)$ is the area of $(d-1)$-dimensional unit sphere, and we denote the total energy density by
		\begin{equation}
			e=\frac12(\d_r\sqrt{\rho})^2+\frac12\Lambda^2+f(\rho).
		\end{equation}
		
		Furthermore for spherically symmetric solutions, the higher order functional $I(t)$, introduced by \eqref{eq:higher_md}, has the form 
		\begin{align}\label{eq:higher_s}
			I(t)=C(d)\int_{0}^\infty r^{d-1}\left[\lambda^2+(\d_t\sqrt{\rho})^2\right] dr,
		\end{align}
		where the function $\lambda$ is implicitly given by the identity
		\begin{equation}\label{eq:lambda_s}
			\sqrt{\rho}\lambda=-\frac14(\d_r^2\rho+\frac{d-1}{r}\d_r\rho)+e+p(\rho).
		\end{equation}
	\end{prop}

	\subsection{Wave function lifting for spherically symmetric data}
	In this section we will establish the wave function lifting method for multi-dimensional spherically symmetric hydrodynamic data. More precisely, the first proposition of this section concerns $(\sqrt\rho(r),\Lambda(r))$ for $r\in\R_+=(0,\infty)$ with finite total mass and energy such that $\Lambda=0$ a.e. on the vacuum region $\{\rho=0\}$, and we will show that there exists a wave function $\psi\in H^1(\R_+,r^{d-1}dr)$ associated to $(\sqrt\rho,\Lambda)$ in the sense of the Definition \ref{def:ass}, which can be equivalently written 
	\begin{equation}\label{eq:polar_s}
		\sqrt\rho=|\psi|,\quad \Lambda=\IM(\bar\phi\d_r\psi)
	\end{equation}
	for spherically symmetric hydrodynamic state, where $\phi\in P(\psi)$ is a polar factor of $\psi$.
	
	\begin{prop}\label{prop:lift_s}
		Let $(\sqrt\rho,\Lambda)$ be a hydrodynamic state on $\R_+$ such that $\Lambda=0$ a.e. on the set $\{\rho=0\}$. Let us assume the bounds
		\begin{equation}\label{eq:C1_s}
			\|\sqrt\rho\|_{H^1(\R_+,r^{d-1}dr)}+\|\Lambda\|_{L^2(\R_+,r^{d-1}dr)}\le M_1,
		\end{equation}
		Then there exists a wave function $\psi\in H^1(\R_+,r^{d-1}dr)$ associated to $(\sqrt\rho,\Lambda)$ in the sense of \eqref{eq:polar_s}.
	\end{prop}
	
	\begin{proof}
		
		Let us consider a sequence $\{\delta_n\}$ of Schwartz functions such that $\delta_n(r)>0$ for all $r\in\R^+$ and $\delta_n\to0$ as $n\to\infty$. For instance we may consider $\delta_n(r)=\frac1ne^{-r^2/2}$. We define the following hydrodynamical states 
		\begin{equation*}
			\sqrt{\rho_n}=\sqrt{\rho}+\delta_n, \quad\Lambda_n=\Lambda,
		\end{equation*}
		then we can check that also $(\sqrt{\rho_n}, \Lambda_n)$ satisfy \eqref{eq:C1_s} uniformly in $n\in\N$.
		Furthermore, since $\sqrt{\rho_n}(r)>0$ it is possible to define the velocity field
		\begin{equation*}
			v_n=\frac{\Lambda_n}{\sqrt{\rho_n}}.
		\end{equation*}
		Notice that since $\sqrt{\rho_n}$ is uniformly bounded away from zero on compact intervals we have $v_n\in L^1_{loc}(\R^+)$, hence it makes sense to define the phase function for $r>0$ by 
		\begin{equation*}
			S_n(r)=\int_{r_0}^r v_n(s)\,ds,
		\end{equation*}
		where $r_0>0$ is a given point, and consequently the radially symmetric wave function 
		\begin{equation*}
			\psi_n(r)=\sqrt{\rho_n}(r)e^{iS_n(r)}.
		\end{equation*}
		We can now show that the sequence $\psi_n$ has a limit $\psi\in H^1(\R^+,r^{d-1}dr)$ which satisfies the polar factorization. Indeed, since
		$\d_r\psi_n=e^{iS_n}(\d_r\sqrt{\rho_n}+i\Lambda_n)$, we have
		\begin{equation*}
			\|\psi_n\|_{H^1(\R^+,r^{d-1}dr)}^2=\|\sqrt{\rho_n}\|_{H^1(\R^+,r^{d-1}dr)}^2+\|\Lambda_n\|_{L^2(\R^+,r^{d-1}dr)}^2\leq C,
		\end{equation*}
		thus, up to subsequences, $\psi_n\rightharpoonup\psi$ in $H^1(\R^+,r^{d-1}dr)$. On the other hand, we also have $\sqrt{\rho_n}\to\sqrt{\rho}$ in $H^1(\R^+,r^{d-1}dr)$, $\Lambda_n=\Lambda$ and moreover $e^{iS_n}\rightharpoonup\phi$ weakly* in $L^\infty(\R^+)$, for some $\phi\in L^\infty(\R^+)$. It is straightforward to check that $\phi$ is a polar factor for $\psi$, since $\psi_n=\sqrt{\rho_n}e^{iS_n}\rightharpoonup\sqrt{\rho}\phi$. Furthermore
		\begin{equation*}
			\d_r\psi_n=e^{iS_n}\left(\d_r\sqrt{\rho_n}+i\Lambda_n\right)\rightharpoonup\phi\left(\d_r\sqrt{\rho}+i\Lambda\right),
		\end{equation*}
		so that $\d_r\psi=\phi\left(\d_r\sqrt{\rho}+i\Lambda\right)$ and hence $(\sqrt{\rho}, \Lambda)$ are the hydrodynamical states associated to $\psi$.
	\end{proof}
	
	We should emphasise that, different from the situation we considered in Section \ref{sect:lift_2d}, the possible existence of large vacuum prevents the wave functions obtained in Proposition \ref{prop:lift_s} to be unique at $H^1$ level. Indeed, the next lemma shows that one can introduce arbitrary constant phase shift in each connected component of the non-vacuum region $\{\rho>0\}$ without breaking the $H^1$ regularity of the wave function, and it leads to the non-uniqueness of lifted wave functions, see the Remark \ref{rmk:non_uniq} (similar result is discussed in more details in Section 3 in \cite{AMZ}).
	
	Here we first recall some elementary properties of 1-dimensional Sobolev functions that will be used soon. Let $g$ be a function in $H^1(a,b)$ and $H^1(b,c)$, then $g\in H^1(a,c)$ if and only if $g$ is continuous at point $b$, and in this case we have
	\[
	\|\d_r g\|_{L^2(a,c)}=\|\d_r g\|_{L^2(a,b)}+\|\d_r g\|_{L^2(b,c)}.
	\]
	Then let us consider $g\in H^1(\Omega)$, where $\Omega\subset\R_+$ is an open set. The continuity of $g$ allows us to decomposed the set $\{r;g(r)\ne0\}$ into disjoint open intervals, i.e. 
	\begin{equation}\label{eq:poscomp}
		\{r;g\ne0\}=\cup_{j}(a_j,b_j), \quad g(a_j)=g(b_j)=0.
	\end{equation}
	
	\begin{lem}\label{lemma:H1ext}
		Let $g\in H^1(\Omega)$ and let $\Theta\in L^\infty(\Omega)$ be a piecewise constant phase shift given by the formula
		\begin{equation}\label{eq:theta}
			\Theta=\exp\left(i\sum_j\theta_j\mathbf{1}_{(a_j,b_j)}\right),\quad \theta_j\in[0,2\pi),
		\end{equation}
		where $(a_j,b_j)$'s are the components of $\{g\ne 0\}$ as in \eqref{eq:poscomp}. 
		Then we have $\Theta g\in H^1(\Omega)$, and 
		\begin{equation}\label{eq:dtheta}
			\d_x (\Theta g)=\Theta \d_x g.
		\end{equation} 
	\end{lem}
	
	\begin{proof}
		The proof of the lemma follows a standard argument of weak derivative. We take $\eta\in C_c^\infty(\Omega)$ to be a test function and consider the weak derivative $\d_r(\Theta g)$:
		\begin{align*}
			\int_\Omega\eta\d_r (\Theta g)\,dr=&-\int_\Omega \Theta g\d_r\eta\,dr=-\sum_j\int_{a_j}^{b_j}e^{i\theta_j}g\d_r\eta \,dr.
		\end{align*}
		On each interval $(a_j,b_j)$ by integration by parts and $g(a_j)=g(b_j)=0$, we obtain
		\[
		\int_\Omega\eta\d_r (\Theta g)\,dx=\sum_j\int_{a_j}^{b_j}e^{i\theta_j}\eta \d_r g \,dr.
		\]
		Furthermore the vanishing derivative Lemma \ref{lemma:LL} implies $\d_r g=0$ a.e. outside $\{r;g(r)\ne0\}=\underset{j}{\cup}(a_j,b_j)$, therefore we can conclude
		\[
		\int_\Omega\eta\d_r (\Theta g)\,dr=\int_\Omega\eta \Theta\d_rg\,dx,
		\]
		namely $\d_r (\Theta g)=\Theta\d_rg\in L^2(\Omega)$, which finishes the proof of the lemma.
	\end{proof}
	
	\begin{rem}\label{rmk:non_uniq}
		As pointed out before, a direct consequence of Lemma \ref{lemma:H1ext} is the non-uniqueness of wave function lifting for  given hydrodynamic state at $H^1$ level, due to the arbitrary phase shifts allowed on the components of the non-vacuum regions. More precisely, let $\psi\in H^1$ be a wave function associated to $(\sqrt\rho,\Lambda)$ in the sense of \eqref{eq:polar_s}, then for any piecewise phase shift function $\Theta$ of the form \eqref{eq:theta}, by formula \eqref{eq:dtheta} it is straightforward to check that $\Theta\psi\in H^1$ is another wave function associated to the same hydrodynamic state, whose polar factors are $P(\Theta\psi)=\Theta P(\psi)$.
	\end{rem}
	
	The next goal of this section is to prove the wave function lifting at $H^2$ level for spherically symmetry state satisfying the conditions in Theorem \ref{thm:glob2_s}. By using standard coordinate change, we see that to prove a spherically symmetry $H^1$ wave function $\psi$ has higher regularity in $H^2(\R^d)$, we need to show $\d_r\psi\in H^1_{loc}(\R_+)$ with the bound
	\[
	\|\d_r^2\psi+\frac{d-1}{r}\d_r\psi\|_{L^2(\R_+,r^{d-1}dr)}\leq C.
	\]
	We should emphasise the important fact that, by the 1-dimensional Sobolev embedding, the requirement $\d_r\psi\in H^1_{loc}(\R_+)$ implies $\d_r\psi$ to be a continuous function on $\R_+$. However this is in general not true for the lifted wave functions $\psi$ obtain in Proposition \ref{prop:lift_s}, since $\d_r\psi$ may experience jump discontinuity at vacuum boundaries. To overcome this difficulty, we will construct a well-design piecewise constant phase shift $\Theta$ of the form \eqref{eq:theta} to balance the possible phase jumps of $\d_r\psi$ at vacuum boundaries, such that $\tilde\psi=\Theta\psi$ has higher regularity in $H^2$ space. 
	
	\begin{prop}\label{prop:lift2_s}
		Let $(\sqrt\rho,\Lambda)$ be a hydrodynamic state on $\R_+$ as given in Proposition \ref{prop:lift_s}. If we further assume the conditions
		\begin{itemize}
			\item $\d_r^2\rho\in L^1_{loc}(\R_+)$, $\d_r J\in L^1_{loc}(\R_+)$;
			\item the energy density $e:=\frac12(\d_r\sqrt{\rho})^2+\frac12\Lambda^2+f(\rho)$ is continuous;
			\item the hydrodynamic state satisfies
			\begin{equation}\label{eq:C2_s}
				\begin{aligned}
					\|\left[\frac{\Lambda^2}{\sqrt{\rho}}\right.-&\left.(\d_r+\frac{d-1}{r})\d_r\sqrt\rho\right]\mathbf{1}_{\{\rho>0\}}\|_{L^2(\R_+,r^{d-1}dr)}\\
					&+\|\frac{\d_r J+(d-1)J/r}{\sqrt{\rho}}\mathbf{1}_{\{\rho>0\}}\|_{L^2(\R_+,r^{d-1}dr)}\leq M_2.
				\end{aligned}
			\end{equation}
		\end{itemize}
		then there exists a wave function $\psi\in H^1(\R_+,r^{d-1}dr)$ associated to $(\sqrt\rho,\Lambda)$ in the sense of \eqref{eq:polar_s}, satisfying $\d_r\psi\in H^1_{loc}(\R_+)$ and the bound
		\begin{equation}\label{eq:B2_s}
			\|\d_r^2\psi+\frac{d-1}{r}\d_r\psi\|_{L^2(\R_+,r^{d-1}dr)}\leq C(M_1,M_2).
		\end{equation}
	\end{prop}
	
	The energy bound $\eqref{eq:C1_s}$ of finite mass and energy implies the continuity of the density $\rho$, therefore as in \eqref{eq:poscomp} we can decompose the non-vacuum region $\{\rho>0\}$ as countable disjoint open intervals, namely
	\begin{equation}\label{eq:nonvac}
		\{\rho>0\}=\underset{j}{\cup}(a_j,b_j).
	\end{equation}
	Before proving Proposition \ref{prop:lift2_s}, we first state the following technical lemma.
	
	\begin{lem}\label{lemma:H2}
		Let us assume $(\sqrt\rho,\Lambda)$ is a hydrodynamic state as given in Proposition \ref{prop:lift2_s}, and let $\psi\in H^1(\R_+,r^{d-1}dr)$ be a wave function associated to $(\sqrt\rho,\Lambda)$. Then the identity
		\begin{equation}\label{eq:dx2psi}
			\d_r^2\psi=\left[\d^2_r\sqrt\rho-\frac{\Lambda^2}{\sqrt\rho}+i\frac{\d_rJ}{\sqrt\rho}\right]\phi
		\end{equation}
		holds true in $L^1_{loc}(a_j,b_j)$ on all the connected components $(a_j,b_j)$ of $\{\rho>0\}$, with $\phi\in P(\psi)$. As a consequence $\d_r\psi\in H^1(a_j,b_j)$ if $a_j>0$, and $\d_r\psi\in H^1_{loc}(a_j,b_j)$ if $a_j=0$.
		
		Moreover, if $\tilde r>0$ is an accumulation point of vacuum boundaries $\underset{j}{\cup}\{a_j,b_j\}$, then it follow $\d_r\psi(\tilde r)=0$.
	\end{lem}
	
	\begin{proof}
		Consider $(\sqrt\rho,\Lambda)$ to be a hydrodynamic state as given in Proposition \ref{prop:lift2_s}, and let $\psi\in H^1(\R_+,r^{d-1})$ be a wave function associated to $(\sqrt\rho,\Lambda)$. By using the polar factorization, it follows
		\[
		\psi=\sqrt\rho\phi,\;\d_r\psi=(\d_r\sqrt\rho+i\Lambda)\phi
		\]
		for some $\phi\in P(\psi)$. Since $\psi$ is continuous and $|\psi|>0$ on $(a_j,b_j)$, the polar factor $\phi\in P(\psi)$ is uniquely defined, $\phi=\frac{\psi}{|\psi|}$, and a direct computation gives
		\begin{align*}
			\d_r\phi=\frac{1}{|\psi|}\left(\d_r\psi-\d_r|\psi|\frac{\psi}{|\psi|}\right)
			=\frac{i\Lambda}{\sqrt\rho}\phi.
		\end{align*}
		Using the identities above and the definition of hydrodynamic states, we obtain
		\begin{align*}
			\left[\d_r^2\sqrt\rho-\frac{\Lambda^2}{\sqrt{\rho}}+i\frac{\d_rJ}{\sqrt\rho}\right]\phi=&\left(\d_r^2\sqrt\rho+i\d_r\Lambda+i\d_r\sqrt\rho\frac{\Lambda}{\sqrt\rho}-\frac{\Lambda^2}{\sqrt\rho}\right)\phi\\
			=&\d_r\left[(\d_r\sqrt\rho+i\Lambda)\phi\right]=\d_r(\d_r\psi).
		\end{align*}
		Thus we prove the identity \eqref{eq:dx2psi}, and by using the bounds \eqref{eq:C1_s} and \eqref{eq:C2_s}, we obtain
		\begin{equation}\label{eq:locbd_1}
			\begin{aligned}
				\|\d_r^2\sqrt\rho-\frac{\Lambda^2}{\sqrt{\rho}}\|_{L^2(a_j,b_j)}\leq & 
				a_j^{\frac{1-d}{2}}\|\frac{\Lambda^2}{\sqrt{\rho}}-(\d_r+\frac{d-1}{r})\d_r\sqrt\rho\|_{L^2((a_j,b_j),r^{d-1}dr)}\\
				&+(d-1)a_j^{-\frac{1+d}{2}}\|\d_r\sqrt\rho\|_{L^2((a,b),r^{d-1}dr)}
				\leq C(M_1,M_2,a_j),
			\end{aligned}
		\end{equation}
		and similarly we have
		\begin{equation}\label{eq:locbd_2}
			\begin{aligned}
				\|\frac{\d_rJ}{\sqrt\rho}\|_{L^2(a_j,b_j)}\leq & 
				a_j^{\frac{1-d}{2}}\|\frac{\d_r J+(d-1)J/r}{\sqrt{\rho}}\|_{L^2((a_j,b_j),r^{d-1}dr)}\\
				&+(d-1)a_j^{-\frac{1+d}{2}}\|\Lambda\|_{L^2((a_j,b_j),r^{d-1}dr)}
				\leq C(M_1,M_2,a_j).
			\end{aligned}
		\end{equation}
		Thus it follows $\d_r\psi\in H^1(a_j,b_j)$ if $a_j>0$. Similarly in the case $a_j=0$, we can show $\d_r\psi\in H^1_{loc}(a_j,b_j)$.
		
		Now let us assume $\tilde r>0$ is an accumulation point of the vacuum boundaries $\cup_j\{a_j,b_j\}$, namely there exists a sequence of intervals $\{(a_{j_k},b_{j_k})\}_k$ such that $(a_{j_k},b_{j_k})\to \tilde r$ as $k\to\infty$. On each $(a_{j_k},b_{j_k})$, we take $r_{j_k}=\frac12(a_{j_k}+b_{j_k})$, and we claim $\d_r\psi(r_{j_k})\to 0$ as $k\to\infty$. On the other hand, by polar factorization the energy density is given by
		\[
		e=\frac12|\d_r\psi|^2+f(|\psi|^2).
		\]
		In particular, the continuity condition of $e$ given in Proposition \ref{prop:lift2_s} implies that $|\d_r\psi|$ is also a continuous function. Thus if the claim holds, by the continuity of $|\d_r\psi|$ we obtain $\d_r\psi(\tilde r)=0$. 
		
		The claim $\d_r\psi(r_{j_k})\to 0$ can be proved by a contradiction. Indeed, we assume that
		$$\underset{k}{\liminf}|\d_r\psi(r_{j_k})|>c>0.$$
		Since $a_{j_k}$ and $b_{j_k}$ are vacuum boundary points, namely $\psi(a_{j_k})=\psi(b_{j_k})=0$, then we have
		\[
		0=\int_{a_{j_k}}^{b_{j_k}}\d_r\psi(s)ds.
		\]
		Moreover, since $(a_{j_k},b_{j_k})\to \tilde r>0$, we can assume $a_{j_k}\sim \tilde r$ for $k$ large, then by the first part of this Lemma, we have $\d_r\psi\in H^1(a_{j_k},b_{j_k})$, so we can decompose the integral as
		\begin{align*}
			0=&\int_{a_{j_k}}^{b_{j_k}}\left[\d_r\psi(r_{j_k})+\int_{r_{j_k}}^{s}\d_r^2\psi(s_1)ds_1\right]ds\\
			=&\d_r\psi(r_{j_k})\delta_{j_k}+\int_{a_{j_k}}^{b_{j_k}}\int_{r_{j_k}}^{s}\d_r^2\psi(s_1)ds_1ds,
		\end{align*}
		where $\delta_{j_k}=b_{j_k}-a_{j_k}$. By using the identity \eqref{eq:dx2psi} and the bounds \eqref{eq:locbd_1},  \eqref{eq:locbd_2}, it follows
		\[
		\|\d_r^2\psi\|_{L^2(a_{j_k},b_{j_k})}\leq \|\d_r^2\sqrt\rho-\frac{\Lambda^2}{\sqrt{\rho}}\|_{L^2(a_{j_k},b_{j_k})}+\|\frac{\d_rJ}{\sqrt\rho}\mathbf{1}_{\{\rho>0\}}\|_{L^2(a_{j_k},b_{j_k})}\leq C(M_1,M_2,\tilde r),
		\]
		then we have
		\[
		\left|\int_{a_{j_k}}^{b_{j_k}}\int_{r_{j_k}}^{s}\d_r^2\psi(s_1)ds_1ds\right|\leq \delta_{j_k}^\frac32\|\d_r^2\psi\|_{L^2(a_{j_k},b_{j_k})}
		\leq \delta_{j_k}^\frac32 C(M_1,M_2,\tilde r).
		\]
		For $k$ large enough one has $|\d_r\psi(r_{j_k})|>c$ and $\delta_{j_k}^\frac12 C(M_1,M_2,\tilde r)<\frac{c}{2}$ since $\delta_{j_k}\to0$ as $k\to\infty$, hence it follows
		\begin{align*}
			0\ge |\d_r\psi(r_{j_k})\delta_{j_k}|-\left|\int_{a_{j_k}}^{b_{j_k}}\int_{r_{j_k}}^{s}\d_r^2\psi(s_1)ds_1ds\right|>\frac{c}{2}\delta_{j_k}>0,
		\end{align*}
		which is a contradiction. 
	\end{proof}
	
	Now we are at the point to prove Proposition \ref{prop:lift2_s}.
	
	\begin{proof}[Proof of Proposition \ref{prop:lift2_s}]
		Let us assume that $(\sqrt\rho,\Lambda)$ is a hydrodynamic state as given in the hypothesis of Proposition. As in \eqref{eq:nonvac}, let us denote
		\begin{equation}\label{eq:non_vac}
			\{\rho>0\}=\cup_{j}(a_j, b_j),\quad \rho(a_j)=\rho(b_j)=0.
		\end{equation}
		Lemma \ref{lemma:H2} shows that on all the connected components $(a_j,b_j)$ of $\{\rho>0\}$ we have $\d_r\psi\in H^1(a_j,b_j)$, and the identity \eqref{eq:dx2psi} holds true almost everywhere on $(a_j,b_j)$. However as discussed before, to obtain a $H^2$ wave function defined on the whole $\R_+$, we need to overcome the possible jump discontinuity at vacuum boundaries. The additionally assumption on the continuity of energy density allows us to provide an well-designed choice of the phase shifts on every connected component and to construct another wave function $\tilde\psi\in H^2$ associated to the same hydrodynamical states. More precisely, starting from $\psi$, we will construct a $\tilde\psi\in H^2_{loc}(\R_+)$ such that $\tilde\psi=\Theta\psi$, where $\Theta$ is a piecewise phase shift of the form
		\begin{equation}\label{eq:sigmarot}
			\Theta=\exp\left(i\sum_j\theta_j\mathbf{1}_{(a_j,b_j)}\right),\quad \theta_j\in[0,2\pi).
		\end{equation} 
		
		Before the construction of $\Theta$, we first give some discussion on the vacuum boundaries $\cup_j\{a_j,b_j\}$. Let us denote the isolated vacuum points by $\{\mathring{a}_j\}$, namely $\rho(\mathring{a}_j)=0$ and $\rho>0$ on $(\mathring{a}_j-\epsilon,\mathring{a}_j+\epsilon)\setminus\{\mathring{a}_j\}$ for some small $\epsilon>0$. Then we define the set
		\[
		W=\{\rho>0\}\underset{j}{\cup}\{\mathring{a}_j\}.
		\]
		By the continuity of $\rho$ and the definition of $\mathring{a}_j$, it is straightforward to see that $W$ is an open set, consequently it can be represented as disjoint open intervals
		\[
		W=\underset{j}{\cup}(\bar{a}_j,\bar{b}_j),
		\]
		where $\bar{a}_j,\bar{b}_j\in\cup_j\{a_j,b_j\}\setminus\cup_j\{\mathring{a}_j\}$. The set $\cup_j\{a_j,b_j\}\setminus\cup_j\{\mathring{a}_j\}$ consists of the following two types of vacuum boundary points. The first type is the boundary of large vacuum with positive measure, namely $|\psi|^2=\rho \equiv 0$ on $(\bar{a}_j-\epsilon,\bar{a}_j)$ or on $(\bar{b}_j,\bar{b}_j+\epsilon)$ for some $\eps>0$, and by the continuity of $|\d_r\psi|$ we have $|\d_r\psi(\bar{a}_j)|=|\d_r\psi(\bar{b}_j)|=0$. The second type is the accumulation point of vacuum boundaries, where $|\d_r\psi|$ also vanishes as proved in Lemma \ref{lemma:H2}. 
		
		Let us fix an interval $(\bar{a}_j,\bar{b}_j)$ and denote $\{\mathring{a}_{j_k}\}_{k=K_1}^{K_2}$ the vacuum points in $(\bar{a}_j,\bar{b}_j)$. By our construction, the only possible accumulation points of $\{\mathring{a}_{j_k}\}$ are $\bar{a}_j$ and $\bar{b}_j$, hence we can assume $\mathring{a}_{j_k}<\mathring{a}_{j_{k+1}}$ for all $K_1\leq k\leq K_2$. 
		
		The phase shift function $\Theta$ on $(\bar{a}_j,\bar{b}_j)$ is constructed through the following strategy. We start from a fixed $(\mathring{a}_{j_0},\mathring{a}_{j_1})$ and set $\theta_{j_0}=0$, namely $\Theta=1$ on $(\mathring{a}_{j_0},\mathring{a}_{j_1})$. To extend the definition of $\Theta$ to $(\mathring{a}_{j_1},\mathring{a}_{j_2})$, we notice that the intervals $(\mathring{a}_{j_k},\mathring{a}_{j_{k+1}})$ are connected components of the set $\{\rho>0\}$. Thus by Lemma \ref{lemma:H2} we have the $H^1$ regularity of $\d_r\psi$ on $(\mathring{a}_{j_0},\mathring{a}_{j_1})$ and $(\mathring{a}_{j_1},\mathring{a}_{j_2})$, which implies the the left and right side limits $\d_r\psi(\mathring{a}_{j_1}^-)$ and $\d_r\psi(\mathring{a}_{j_1}^+)$ exist. On the other hand, the continuity of the energy density $e$ implies $|\d_r\psi|=\sqrt{2e-f(\rho)}$ is continuous, hence we have $|\d_r\psi(\mathring{a}_{j_1}^-)|=|\d_r\psi(\mathring{a}_{j_1}^+)|$. Therefore we can choose a $\theta_{j_1}\in [0,2\pi)$ such that $(\Theta\cdot\d_r\psi)(\mathring{a}_{j_1}^-)=e^{i\theta_{j_1}}\d_r\psi(\mathring{a}_{j_1}^+)$ ($\theta_{j_1}=0$ if $|\d_r\psi|$ vanish at $\mathring{a}_{j_1}$), and we define $\Theta=e^{i\theta_{j_1}}$ on $(\mathring{a}_{j_1},\mathring{a}_{j_2})$. By repeating this process inductively, we extend the definition of $\Theta$ to the whole $(\bar{a}_j,\bar{b}_j)$. Moreover we have $\Theta\d_r\psi\in H^1(\mathring{a}_{j_k},\mathring{a}_{j_k})$ and our construction ensures the continuity of $\Theta\d_r\psi$ at all point vacuum $\mathring{a}_{j_k}$. Thus by the elementary property of Sobolev functions, it follows $\Theta\d_r\psi\in H^1(\bar{a}_j,\bar{b}_j)$ (or $H^1_{loc}(\bar{a}_j,\bar{b}_j)$ if $\bar{a}_j=0$), and by Lemma \ref{lemma:H1ext} and Lemma \ref{lemma:H2} we have the identity
		\begin{equation}\label{eq:401}
			\d_r(\Theta\d_r\psi)=\left[\d^2_r\sqrt\rho-\frac{\Lambda^2}{\sqrt\rho}+i\frac{\d_rJ}{\sqrt\rho}\right]\Theta\phi\quad a.e.\;on\;(\bar{a}_j,\bar{b}_j),\quad \phi\in P(\psi).
		\end{equation}
		
		To this point we have constructed the phase shift function $\Theta$ on $W=\underset{j}{\cup}(\bar{a}_j,\bar{b}_j)$, and we extend $\Theta$ to the whole $\R_+$ by definition $\Theta=1$ on $W^c$. It is straightforward to see $\Theta$ has the form \eqref{eq:sigmarot} as required, and we define $\tilde\psi=\Theta\psi$. By Lemma \ref{lemma:H1ext} we have $\d_r\tilde\psi=\Theta\d_r\psi$, then it is straightforward to check $\tilde\psi\in H^1(\R_+,r^{d-1})$ is also a wave function associated to $(\sqrt\rho,\Lambda)$ with the polar factorization 
		\begin{equation}
			\d_r\tilde\psi=(\d_r\sqrt\rho+i\Lambda)\Theta\phi.
		\end{equation}
		Now we claim $\d_r\tilde\psi=\Theta\d_r\psi\in H^1_{loc}(\R_+)$. Let $\eta\in C_c^\infty(\R_+)$ be an arbitrary test function supported in $(\epsilon,\infty)$, then we have
		\[
		\int_{\R_+}\eta\,\d_r(\Theta\d_r\psi)dr\coloneqq-\int_{\R_+}(\Theta\d_r\psi)\d_r\eta \,dr.
		\]
		By Lemma \ref{lemma:LL} it follows $\d_r\psi=0$ a.e. outside the set $W$, therefore
		\[
		\int_{\R_+}\eta\,\d_r(\Theta\d_r\psi)dr=-\int_{W}(\Theta\d_r\psi)\d_r\eta\,dr=-\underset{j}{\sum}\int_{\bar{a}_j}^{\bar{b}_j}(\Theta\d_r\psi)\d_r\eta\,dr.
		\]
		Since $\Theta\d_r\psi\in H^1(\bar{a}_j,\bar{b}_j)$ and we can write
		\[
		\int_{\bar{a}_j}^{\bar{b}_j}(\Theta\d_r\psi)\d_r\eta\,dr=-\int_{\bar{a}_j}^{\bar{b}_j}\eta\,\d_r(\Theta\d_r\psi)dr+\eta(\bar{b}_j)(\Theta\d_r\psi)(\bar{b}_j)-\eta(\bar{a}_j)(\Theta\d_r\psi)(\bar{a}_j),
		\]
		where the boundary terms vanish since $\d_r\psi(\bar{a}_j)=\d_r\psi(\bar{b}_j)=0$ if $\bar{a}_j>0$ as we prove before, or we have $\eta(\bar{a}_j)=0$ if $\bar{a}_j=0$. By using the identity \eqref{eq:401} and the bound \eqref{eq:C2_s}, we obtain
		\begin{align*}
			\left|\int_{\R_+}\eta\,\d_r(\Theta\d_r\psi)dr\right|\leq & \|\eta\|_{L^2(\R_+)}\left(\|\d^2_r\sqrt\rho-\frac{\Lambda^2}{\sqrt\rho}\|_{L^2(W\cap(\epsilon,\infty))}+\|\frac{\d_rJ}{\sqrt\rho}\|_{L^2(W\cap(\epsilon,\infty))}\right)\\
			\leq & C(M_1,M_2,\epsilon)\|\eta\|_{L^2(\R_+)}.
		\end{align*}
		Thus we conclude $\d_r\tilde\psi=\Theta\d_r\psi\in H^1_{loc}(\R_+)$. Moreover by Lemma \ref{lemma:LL}, we have $\d_r(\Theta\d_r\psi)=0$ a.e. on the set $\{\rho>0\}=\{\psi=0\}$, therefore the identity
		\begin{equation}\label{eq:402}
			\d_r^2\tilde\psi=\d_r(\Theta\d_r\psi)=\left[\d^2_r\sqrt\rho-\frac{\Lambda^2}{\sqrt\rho}+i\frac{\d_rJ}{\sqrt\rho}\right]\mathbf{1}_{\{\rho>0\}}\Theta\phi
		\end{equation}
		holds true a.e. on $\R_+$. To conclude the proof of Proposition, we only need to show the bound \eqref{eq:B2_s} for the wave function $\tilde\psi$, which is a direct consequence of the identity \eqref{eq:402}, polar factorization and the bound \eqref{eq:C2_s}.
	\end{proof}
	
	As a conclusion of this section, we can prove Theorem \ref{thm:glob_s} and Theorem \ref{thm:glob2_s}, namely the global existence results of spherically symmetric weak solutions to \eqref{eq:QHD_md}, by using Proposition \ref{prop:lift_s}, Proposition \ref{prop:lift2_s}, the well-posedness of the NLS equation and Theorem \ref{thm:QHD_old}. The proof follows exactly the same argument as the proof of Theorem \ref{thm:glob_md} and Theorem \ref{thm:glob2_md} in Section \ref{sect:lift_2d} and we omit it here.

	\section{Dispersive estimates and bounds on $I(t)$}\label{sect:apri_md}
	
	In this section we will collect some a priori estimates satisfied by finite energy weak solutions to $\eqref{eq:QHD_md}$. If we restricted our analysis to Schr\"odinger-generated solutions - like the ones constructed in the global existence theorems in this paper - then the dispersive estimates inherited by the NLS dynamics would yield a wide range of information. For general solutions that is not the case; the quasi-linear nature of system \eqref{eq:QHD_md} prevents to successfully exploit semigroup techniques to infer suitable smoothing estimates. However we are still able to prove some a priori estimates for general solutions to QHD system in genuine hydrodynamic approaches.
	
	We first define the functional
	\begin{equation}\label{eq:pc_en_md}
		H(t)=\int_{\R^d}\frac{|x|^2}{2}\rho(t, x)\,dx-t\int_{\R^d}x\cdot J(t, x)\,dx+t^2E(t),
	\end{equation}
	where the total energy $E(t)$ is given by \eqref{eq:en_QHD_md}.
	The functional $H(t)$ is the hydrodynamic analogue of the pseudo-conformal energy for NLS solutions \cite{GV}, see also \cite{HNT1, HNT2}. Furthermore, the functional $H(t)$ can also be expressed as
	\begin{equation}\label{eq:pc_en_2_md}
		H(t)=\int_{\R^d}\frac{t^2}{2}|\nabla\sqrt{\rho}|^2+\frac{t^2}{2}|\Lambda-\frac{x}{t}\sqrt{\rho}|^2+t^2f(\rho)\,dx.
	\end{equation}
	Thus roughly speaking the functional $H(t)$ measures the time behaviour of the strength of the quantum interaction $|\nabla\sqrt\rho|^2$, the internal energy $f(\rho)$ and the distance between $\Lambda$ and a rarefaction wave $\frac{x}{t}\sqrt\rho$.
	
	\begin{prop}
		Let $(\rho, J)$ be a finite energy weak solution to \eqref{eq:QHD_md} as in the Definition \ref{def:FEWS} such that the total energy $E(t)$ given by \eqref{eq:en_QHD_md} is non-increasing, and 
		\begin{equation}\label{eq:fin_var_md}
			\int_{\R^d}|x|^2\rho_0(x)\,dx<\infty.
		\end{equation}
		Then we have
		\begin{equation*}
			H(t)+d\int_0^t\int_{\R^d}s\rho \left[f'(\rho)-(\frac{2}{d}+1)f(\rho)\right]dxds\leq\int_{\R^d}\frac{|x|^2}{2}\rho_0(x)\,dx.
		\end{equation*}
	\end{prop}
	\begin{proof}
		Since $(\rho,J)$ are only assumed to be a weak solution to \eqref{eq:QHD_md}, to rigorously prove the proposition we need to choose suitable test functions $\eta,\zeta\in C^\infty_c([0,T)\times\R^d)$ in the Definition \ref{def:FEWS}, which are obtained by introducing the following cut-off functions. The spacial cut-off functions is given by
		\[
		\Phi_R(x)=\frac{|x|^2}{2}\chi_R^2(|x|),\quad R>0,
		\]
		where $\chi_R(|x|)=\chi(|x|/R)$, and $\chi\in C^\infty_c(\R)$ such that $\chi(|x|)\equiv 1$ for $|x|\leq 1$, and $\textrm{supp}\chi\subset\{|x|\leq 2\}$. On the other hand, let $\{\tilde\chi_\eps\}_{\eps>0}\subset C^\infty(\R)$ be a sequence of convex functions such that
		\[
		\tilde\chi_\eps(s)=\begin{cases}
			0 &,\ s<-\eps\\
			s &,\ s>\eps
		\end{cases}.
		\]
		It is clear to see that $\{\tilde\chi_\eps''\}$ is a sequence of mollifiers, namely $\|\tilde\chi_\eps''\|_{L^1}=1$, $\textrm{supp}\tilde\chi_\eps''\in\{|s|\leq\eps\}$, and $\{\tilde\chi_\eps''\}$ approximates the Dirac-delta in the space of measure as $\eps\to 0$.

		We first prove that the momentum of inertia increases at most quadratically in time under the initial assumption \eqref{eq:fin_var_md}, namely
		\begin{equation}\label{eq:fin_var_2}
			\int_{\R^d} |x|^2\rho(t,x)dx\leq C t^2+\int_\R |x|^2\rho_0(x)dx.
		\end{equation}
		For $\epsilon<t<T-\epsilon$, by choosing $\eta(s,x)=\tilde\chi_\eps'(t-s)\Phi_R(x)$ in \eqref{eq:QHD_cty}, it follows that
		\begin{equation}\label{eq:prop22_3}
			-\int_0^T\int_{\R^d}\tilde\chi_\eps''(t-s)\Phi_R(x)\rho+\tilde\chi_\eps'(t-s)\nabla\Phi_R(x)\cdot J\,dxds+\int_{\R^d}\tilde\chi_\eps'(t)\Phi_R(x)\rho_0(x)\,dx=0.
		\end{equation}
		By our choice of $\tilde\chi_\eps$, we have $\tilde\chi_\eps\to s\mathbf{1}_{\{s>0\}}$ a.e. $s\in\R$ and $\tilde\chi_\eps'(t)=1$ for $t>\eps$. Furthermore $\{\tilde\chi_\eps''\}$ is a sequence of  mollifiers. Therefore by the property of mollifiers and dominated convergence theorem, as $\epsilon\to 0$ in \eqref{eq:prop22_3} we obtain that
		\begin{equation}\label{eq:prop22_4}
			-\int_{\R^d}\Phi_R(x)\rho(t,x)dx+\int_0^t\int_{\R^d}\nabla\Phi_R(x)\cdot J\,dxds+\int_{\R^d}\Phi_R(x)\rho_0(x)\,dx=0,
		\end{equation}
		for a.e. $t\in(0,T)$. Let us define  
		\[
		V_R(t)=\int_{\R^d} \Phi_R(x)\rho(t,x)dx,
		\] 
		then by our choice of $\Phi_R$ and the finite energy, we have
		\begin{align*}
			\int_{\R^d}\nabla\Phi_R(x)\cdot J\,dx=&\int_{\R^d}(x\,\chi^2_R+|x|^2\chi_R\nabla\chi_R)\cdot(\sqrt\rho\Lambda)\,dx\\
			\leq & \|\chi_R+|x||\nabla\chi_R|\|_{L^\infty(\R^d)}\|\Lambda\|_{L^2(\R^d)}\left(\int_{\R^d} |x|^2\chi_R^2(x)\rho\,dx\right)^\frac12\\
			\leq & C \left(\int_{\R^d} \Phi_R(x)\rho\,dx\right)^\frac12=C\, V_R^\frac12(t).
		\end{align*}
		By substituting the previous inequality into \eqref{eq:prop22_4} and using Gronwall's argument, it follows that
		\[
		V_R(t)\leq C t^2+V_R(0),
		\]
		which implies \eqref{eq:fin_var_2} as $R\to\infty$ . As a direct consequence of \eqref{eq:fin_var_2}, for a.e. $0<t<T$ we have 
		\[
		\int_{\R^d} |x\cdot J(t,x)|dx\leq \|\Lambda(t)\|_{L^2(\R^d)}(\left(\int_{\R^d} |x|^2\rho(t)dx\right)^\frac12<\infty,
		\]
		and by dominated convergence theorem
		\begin{equation}\label{eq:conv_xJ}
			\int_0^t\int_{\R^d}\nabla\Phi_R(x)\cdot J\,dxds\to\int_0^t\int_{\R^d}x\cdot J\,dxds.
		\end{equation}
		
		Next we choose $\zeta_1(s,x)=\tilde\chi_\eps(t-s)\nabla\phi_R(x)$ and $\zeta_2(s,x)=\tilde\chi'_\eps(t-s)\nabla\phi_R(x)$ in
		\eqref{eq:QHD_mom}. As before, by letting $\eps\to 0$ and using the properties of $\{\chi_\eps\}$, the weak formulation of the momentum equation \eqref{eq:QHD_mom} implies
		\begin{equation}\label{eq:prop22_6}
			\begin{aligned}
				-\int_0^t\int_{\R^d}\nabla\phi_R(x)J+&(t-s)\nabla^2\phi_R(x):(\Lambda\otimes\Lambda+\nabla\sqrt{\rho}\otimes \nabla\sqrt\rho+p(\rho)Id)\\
				&-\frac{1}{4}(t-s)\rho\triangle^2\phi_R(x)dxds+\int_{\R^d}t\,\nabla\phi_R(x)\cdot J_0(x)\,dx=0,
			\end{aligned}
		\end{equation}
		and
		\begin{equation}\label{eq:prop22_8}
			\begin{aligned}
				-\int_{\R^d}\nabla\phi_R(x)\cdot J(t,x)dx+&\int_0^t\int_{\R^2}\nabla^2\phi_R(x):(\Lambda\otimes\Lambda+\nabla\sqrt{\rho}\otimes\nabla\sqrt\rho+p(\rho)Id)\\
				&-\frac{1}{4}\rho\triangle^2\phi_R(x)dxds+\int_{\R^d}\nabla\phi_R(x)J_0(x)\,dx=0.
			\end{aligned}
		\end{equation}
		Then $t\times\eqref{eq:prop22_8}-\eqref{eq:prop22_6}-\eqref{eq:prop22_4}$ gives
		\begin{equation}\label{eq:prop22_9}
			\begin{aligned}
				\int_{\R^d}\phi_R(x)\rho(t,x)dx&-t\int_{\R^d}\nabla\phi_R(x)\cdot J(t,x)dx\\
				&+\int^t_0\int_{\R^d} s\nabla^2\phi_R(x):(\Lambda\otimes\Lambda+\nabla\sqrt{\rho}\otimes\nabla\sqrt\rho+p(\rho)Id)\\
				&-\frac{s}{4}\rho\triangle^2\phi_R(x)dxds-\int_{\R^2}\phi_R(x)\rho_0(x)dx=0.
			\end{aligned}
		\end{equation}
		By the definition of $\phi_R$, we have
		\[
		\nabla^2\phi_R(x)=\chi_R^2(x)\mathbb{I}_d+x\otimes\nabla(\chi_R^2)+\nabla(\chi_R^2)\otimes x+\frac{x^2}{2}\nabla^2(\chi_R^2),
		\]
		where $\mathbb{I}_d$ is the identity matrix on $\R^d$, and it is straightforward to see that $\|\nabla^2\phi_R(x)\|_{L^\infty(\R)}\leq C$ and $\nabla^2\phi_R(x)\to \mathbb{I}_d$ a.e. on $\R^d$ as $R\to\infty$. Therefore by using dominated convergence theorem we obtain
		\begin{align*}
			\int^t_0\int_{\R^d} s\nabla^2\phi_R(x):(\Lambda\otimes\Lambda &+\nabla\sqrt{\rho}\otimes\nabla\sqrt\rho+p(\rho)Id)dxds\\
			=&\int^t_0\int_{\R^d} s(|\Lambda|^2+d\,p(\rho)+|\nabla\sqrt{\rho}|^2)dxds\\
			=&\int^t_0 2s\,E(s)+s\int_{\R^d}(d\,p(\rho)-2f(\rho))dxds,
		\end{align*}
		where $E(s)$ denotes the total energy. Also as $R\to\infty$, we have $\triangle^2\phi_R(x)\sim \mathcal{O}(R^{-2})$, hence by letting $R\to\infty$, the identity \eqref{eq:prop22_9} implies 
		\begin{align*}
			\int_{\R^d}\frac{|x|^2}{2}\rho(t,x)dx-&t\int_{\R^d}x\cdot J(t,x)dx+2\int^t_0 s\,E(s)dx\\
			&=\int_{\R^d}\frac{|x|^2}{2}\rho_0(x)dx+2\int_0^t s\int_{\R^d} 2f(\rho)-d\,p(\rho)dxds.
		\end{align*}
		Since $E(s)$ is non-increasing, we have
		\[
		2\int^t_0 s\,E(s)dx\geq t^2E(t),
		\]
		and we can conclude
		\begin{align*}
			H(t)= & \int_{\R^d}\frac{x^2}{2}\rho(t,x)dx-t\int_{\R^d}x\,J(t,x)dx+t^2E(t)\\
			\leq & \int_{\R^d}\frac{x^2}{2}\rho_0(x)dx+2\int_0^t s\int_{\R^d} (2+d)f(\rho)-\rho\,f'(\rho)dxds,
		\end{align*}
		where we also use $p(\rho)=\rho\,f'(\rho)-f(\rho)$.
	\end{proof}
	
	In the case under our consideration we have $f(\rho)=\frac1\gamma\rho^\gamma$ so that
	\begin{equation}\label{eq:pc_md}
		H(t)+d\left(1-\frac{1+2/d}{\gamma}\right)\int_0^ts\int\rho^\gamma\,dx\leq\int\frac{|x|^2}{2}\rho_0(x)\,dx.
	\end{equation}
	We exploit \eqref{eq:pc_en_2_md}, \eqref{eq:pc_md} in order to infer a dispersive type estimate for solutions to \eqref{eq:QHD_md}.
	\begin{prop}\label{prop:disp_md}
		Let $(\rho, J)$ be a finite energy weak solution to system \eqref{eq:QHD_md} and let us further assume that $\int_{\R^d} |x|^2\rho_0(x)\,dx<\infty$. Then we have
		\begin{equation}\label{eq:disp_pc_md}
			H(t)\lesssim t^{2(1-\sigma)}+\int_{\R^d}\frac{|x|^2}{2}\rho_0(x)\,dx,
		\end{equation}
		where $\sigma=\min\{1, \frac{d}{2}(\gamma-1)\}$. In particular we have
		\begin{equation}\label{eq:disp}
			\int|\nabla\sqrt{\rho}(t)|^2dx\lesssim t^{-2\sigma},
		\end{equation}
		\begin{equation}\label{eq:disp_pr}
			\int\rho^\gamma\,dx\lesssim t^{-2\sigma}.
		\end{equation}
		and 
		\begin{equation}\label{eq:disp_kin}
			\int|\Lambda(t, x)-\frac{x}{t}\sqrt{\rho}(t, x)|^2\,dx\lesssim t^{-2\sigma}.
		\end{equation}
	\end{prop}
	
	\begin{proof}
		For $\gamma>1+\frac{2}{d}$, the inequality \eqref{eq:pc_md}directly implies
		\begin{equation*}
			H(t)\leq\int\frac{|x|^2}{2}\rho_0(x)\,dx
		\end{equation*}
		and \eqref{eq:disp_pc_md} holds for $\sigma=1$.
		Let us now consider the case $1<\gamma\leq 1+\frac{2}{d}$. By defining
		\begin{equation*}
			F(t)=\frac{t^2}{\gamma}\int\rho^\gamma(t, x)\,dx
		\end{equation*}
		it follows \eqref{eq:pc_en_2_md} and \eqref{eq:pc_md} that
		\begin{equation*}
			F(t)\leq H(t)\leq[2+d(1-\gamma)]\int_0^t\frac1s F(s)\,ds+\int\frac{|x|^2}{2}\rho_0(x)\,dx.
		\end{equation*}
		By Gronwall's inequality we then have
		\begin{equation*}
			F(t)\lesssim t^{2+d(1-\gamma)}F(1)+\int\frac{|x|^2}{2}\rho_0(x)\,dx,
		\end{equation*}
		which also implies 
		\begin{equation*}
			\int\rho^\gamma(t, x)\,dx\lesssim t^{d(1-\gamma)}.
		\end{equation*}
		We can now substitute the above estimate into \eqref{eq:pc_en_2_md} in order to obtain
		\begin{equation*}
			H(t)\lesssim\int_0^t s^{1+d(1-\gamma)}\,ds+\int\frac{x^2}{2}\rho_0(x)\,dx\lesssim t^{2-d(\gamma-1)}+\int\frac{|x|^2}{2}\rho_0(x)dx.
		\end{equation*}
	\end{proof}
	
	\begin{rem}
		Proposition \ref{prop:disp_md} shows that the gradient of mass density converges to zero and formally the velocity field asymptotically approaches a rarefaction wave, namely 
		$v(t, x)\sim x/t$ as $t\to\infty$.
		From the conservation of energy we also infer that
		\begin{equation*}
			\lim_{t\to\infty}\frac12\|\Lambda(t)\|_{L^2}^2=E(0),
		\end{equation*}
		i.e. for large times all the energy is transferred to the kinetic part. Moreover by Gagliardo-Nirenberg the decay of $\|\nabla\sqrt{\rho}(t)\|_{L^2}$ implies the following dispersive estimate 
		\begin{equation*}
			\|\sqrt{\rho}(t)\|_{L^p}\lesssim \|\sqrt{\rho}(t)\|_{L^2}^{1-\alpha}\|\nabla\sqrt{\rho}(t)\|_{L^2}^\alpha \lesssim t^{-\alpha\sigma},
		\end{equation*}
		where $0\leq\alpha\leq 1$ such that 
		\begin{equation*}
			\frac{1}{p}=\frac12-\frac{\alpha}{d},
		\end{equation*}
		and $2\leq p <\infty$ for $d=2$, $2\leq p\leq 6$ for $d=3$. Notice that in the case $\gamma\geq1+2/d$ this is still consistent with the dispersive estimate for the free Schr\"odinger equation.
	\end{rem}
	
	In general similar estimates appear in many contexts for the study of evolutionary PDEs, see \cite{Straus, Gl, Sid, Mor} and many others. In particular estimates like \eqref{eq:disp_pc_md} have been considered in classical compressible fluid dynamics \cite{Chem, Ser, Sid, Xin}, in \cite{BD, Pe, IR} for the analogue \eqref{eq:disp_kin} in kinetic theory and for NLS equations \cite{Bar, GV, Gl}.  There is a connected result for a particular class of Euler--Korteweg systems, somehow related to the dispersive logarithmic Schr\"odinger equation ($\gamma=1$), given in \cite{CarHil}. A similar result appeared also in the Appendix of the review \cite{CDS}. Those are also somehow reminiscent of the vector field method \cite{Kl} used to study nonlinear wave equations, see also \cite{FJS} for recent applications.
	In the wave function dynamics context, the fact that solutions to the nonlinear problem disperse as much as the linear solutions gives some information about their asymptotic behavior, see \cite{GV}. Recently an alternative proof, based on interaction Morawetz estimates \cite{Mor, LS}, was developed in order to show asymptotic completeness for mass-supercritical, energy-subcritical NLS equations \cite{CGT, CKSTT, GVQ, PV}. It turns out that such estimates hold also for arbitrary weak solutions to \eqref{eq:QHD_md}. 
	
	In this section we also present the Morawetz-type estimates, in a general form for the Euler-Korteweg system
	\begin{equation}\label{eq:EK_md}
		\left\{\begin{aligned}
			&\d_t\rho+\diver J=0\\
			&\d_tJ+\diver\left(\frac{J\otimes J}{\rho}\right)+\nabla p(\rho)=\rho\nabla\left(\diver(\kappa(\rho)\nabla\rho)-\frac12\kappa'(\rho)|\nabla\rho|^2\right),
		\end{aligned}\right.
	\end{equation}
	where $\kappa:(0,\infty)\to(0,\infty)$ is a smooth function, and by choosing $\kappa(\rho)=\frac{1}{4\rho}$ we recover the QHD system. The term on the right hand side of \eqref{eq:EK_md} describes capillarity effects in diffuse interfaces \cite{Kor}, see also \cite{DS} for its derivation. Moreover, while for the dispersive estimates in Proposition \ref{prop:disp_md} we need to assume the initial mass density to have finite variance \eqref{eq:fin_var_md}, in the next proposition we only need the weak solutions to be finite energy.
	
	\begin{prop}\label{prop:mor_md}
		Let $(\rho, J)$ be a weak solution to the system \eqref{eq:EK_md} such that
		\[
		\|\sqrt\rho\|_{L^\infty(0,T;L^2(\R^d))}+\|\kappa(\rho)^\frac12\nabla\rho\|_{L^\infty(0,T;L^2(\R^d))}+\|\Lambda\|_{L^\infty(0,T;L^2(\R^d))}\leq M_1.
		\]
		We define $K(\rho)=\int_0^\rho s\,\kappa(s)\,ds$ and further assume $K(\rho)$ is a convex function in $\rho$, namely $\kappa(\rho)+\rho\,\kappa'(\rho)\geq 0$ for $\rho\geq 0$. Then we have 
		\begin{equation*}
			\||\nabla|^{\frac{3-d}{2}}\sqrt{\rho K(\rho)}\|_{L^2(\R_t\times\R^d_x)}^2
			+\||\nabla|^{\frac{1-d}{2}}\sqrt{\rho p(\rho)}\|_{L^2(\R_t\times\R^d_x)}^2\lesssim M_1^4.
		\end{equation*}
	\end{prop}
	
	\begin{proof}
		Through a direct computation, the right hand side of the equation for momentum density $J$ in \eqref{eq:EK_md} can be written as 
		\[
		\nabla\diver(\rho\,\kappa(\rho)\nabla\rho)-\diver(\kappa(\rho)\nabla\rho\otimes\nabla\rho)-\frac12\nabla\left[(\kappa(\rho)+\rho\,\kappa'(\rho))|\nabla\rho|^2\right].
		\]
		By the definition of $K(\rho)$ we have 
		\[
		\nabla\diver(\rho\,\kappa(\rho)\nabla\rho)=\nabla\triangle K(\rho),
		\]
		then we can rewrite the system \eqref{eq:EK_md} as
		\begin{equation}\label{eq:EK_md_2}
			\left\{\begin{aligned}
				&\d_t\rho+\diver J=0\\
				&\d_tJ+\diver\left(\frac{J\otimes J}{\rho}+\kappa(\rho)\nabla\rho\otimes\nabla\rho\right)+\nabla p(\rho)=\nabla\triangle K(\rho)-\frac12\nabla\left[(\kappa(\rho)+\rho\,\kappa'(\rho))|\nabla\rho|^2\right]
			\end{aligned}\right.
		\end{equation}
		Now we consider the interaction functional 
		\[
		Q(t)=\int_{\R^{2d}}\rho(t,y)\frac{x-y}{|x-y|}\cdot J(t,x)dxdy.
		\]
		The functional $Q(t)$ is well-defined on $[0,T]$ with bound
		\[
		|Q(t)|\leq \|\rho\|_{L^\infty_tL^1_x}\|J\|_{L^\infty_tL^1_x} \leq M_1^4.
		\]
		By differentiating $Q(t)$ in time using the system \eqref{eq:EK_md_2} and integrating by parts we obtain
		\begin{equation}\label{eq:dtH}
			\begin{aligned}
				\frac{d}{dt}Q(t)=&-\int_{\R^{2d}}J(t,y)\cdot B(x,y)\cdot J(t,x)dxdy+\int_{\R^{2d}}\rho(t,y)\Lambda(t,x)\cdot B(x,y)\cdot \Lambda(t,x)dxdy\\
				&+\int_{\R^{2d}}\rho(t,y)\kappa(\rho)(t,x)\nabla\rho(t,x)\cdot B(x,y)\cdot \nabla\rho(t,x)dxdy\\
				&+\int_{\R^{2d}}\rho(t,y)\frac{d-1}{|x-y|}\left[p(\rho)+\frac12(\kappa(\rho)+\rho\,\kappa'(\rho))|\nabla\rho|^2\right](t,x)dxdy\\
				&-\int_{\R^{2d}}\rho(t,y)\frac{d-1}{|x-y|}\triangle K(\rho)(t,x)dxdy\\
				=&\textrm{I}+\textrm{II}+\textrm{III}+\textrm{IV}+\textrm{V}.
			\end{aligned}
		\end{equation}
		where 
		\begin{align*}
			B(x,y)=\nabla_x\left(\frac{x-y}{|x-y|}\right)=\frac{1}{|x-y|}\left[\mathbb{I}_d-\frac{(x-y)\otimes(x-y)}{|x-y|^2}\right],
		\end{align*}
		where $\mathbb{I}_d$ is the identity matrix. Since $B(x,y)$ is a semi-positive definite matrix and symmetric in $x$ and $y$, we have 
		\begin{align*}
			-\textrm{I}=&\int_{\R^{2d}}J(t,y)\cdot B(x,y)\cdot J(t,x)dxdy\\
			\leq & \frac12 \int_{\R^{2d}}\rho(t,y)\Lambda(t,x)\cdot B(x,y)\cdot \Lambda(t,x)dxdy\\
			&+\frac12\int_{\R^{2d}}\rho(t,x)\Lambda(t,y)\cdot B(x,y)\cdot \Lambda(t,y)dxdy\\
			= & \int_{\R^{2d}}\rho(t,y)\Lambda(t,x)\cdot B(x,y)\cdot \Lambda(t,x)dxdy=\textrm{II}.
		\end{align*}
		Therefore $\textrm{I}+\textrm{II}$ are non-negative, and we also have $\textrm{III}\geq 0$. In the integral $\textrm{IV}$, by our assumption we have
		\[
		\int_{\R^{2d}}\rho(t,y)\frac{d-1}{|x-y|}\left[(\kappa(\rho)+\rho\,\kappa'(\rho))|\nabla\rho|^2\right](t,x)dxdy\geq 0.
		\]
		we treat the remaining part in $\textrm{IV}$ in the following way. By the symmetry in $x$ and $y$, we have 
		\begin{align*}
			\int_{\R^{2d}}\rho(t,y)\frac{d-1}{|x-y|}p(\rho)(t,x)dxdy=&\frac12\int_{\R^{2d}}\rho(t,y)\frac{d-1}{|x-y|}p(\rho)(t,x)dxdy\\
			&+\frac12\int_{\R^{2d}}\rho(t,x)\frac{d-1}{|x-y|}p(\rho)(t,y)dxdy\\
			\geq & \int_{\R^{2d}}\sqrt{\rho\,p(\rho)}(t,y)\frac{d-1}{|x-y|}\sqrt{\rho\,p(\rho)}(t,x)dxdy
		\end{align*}
		The convolution of the kernel $|x-y|^{-1}$ gives 
		\[
		\int_{\R^d}\frac{1}{|x-y|}\sqrt{\rho\,p(\rho)}(t,x)dx=C|\nabla|^{1-d}\sqrt{\rho\,p(\rho)}(t,y)
		\]
		for a constant $C>0$, then we obtain
		\begin{align*}
			\textrm{IV}\geq\int_{\R^{2d}}\rho(t,y)\frac{d-1}{|x-y|}p(\rho)(t,x)dxdy\geq & C\int_{\R^d}\sqrt{\rho\,p(\rho)}(t,y)|\nabla|^{1-d}\sqrt{\rho\,p(\rho)}(t,y)dy\\
			= &C\||\nabla|^{\frac{1-d}{2}}\sqrt{\rho\,p(\rho)}\|_{L^2_x}^2(t).
		\end{align*}
		Similarly the integral $\textrm{V}$ can be written as 
		\[
		\textrm{V}=-\int_{\R^{2d}}\rho(t,y)\frac{d-1}{|x-y|}\triangle K(\rho)(t,x)dxdy=C\int_{\R^{d}}\rho(t,y)|\nabla|^{3-d} K(\rho)(t,y)dy,
		\]
		and again by the symmetry we obtain
		\[
		\textrm{V}\geq C \||\nabla|^{\frac{3-d}{2}}\sqrt{\rho K(\rho)}\|_{L^2_x}^2(t).
		\]
		By summing \textrm{I} to \textrm{V}, we get
		\begin{equation}\label{eq:dtH_2}
			\frac{d}{dt}H(t)\geq C \||\nabla|^{\frac{1-d}{2}}\sqrt{\rho\,p(\rho)}\|_{L^2_x}^2(t)+C \||\nabla|^{\frac{3-d}{2}}\sqrt{\rho K(\rho)}\|_{L^2_x}^2(t),
		\end{equation}
		then integrating \eqref{eq:dtH_2} in time shows
		\[
		\||\nabla|^{\frac{1-d}{2}}\sqrt{\rho\,p(\rho)}\|_{L^2_tL^2_x}^2+C \||\nabla|^{\frac{3-d}{2}}\sqrt{\rho K(\rho)}\|_{L^2_tL^2_x}^2\lesssim H(T)-H(0)\lesssim M_1^4.
		\]
	\end{proof}
	
	In the remaining part of this section we are going to discuss the 2nd order energy functional introduced in \eqref{eq:higher_md} in multi-dimension. As already discussed in Section \ref{sect:prel}, if we consider Schr\"odinger-generated solutions to the QHD system \eqref{eq:QHD_md}, say $\rho=|\psi|^2$, $J=\IM(\bar\psi\nabla\psi)$, for some $H^2$ solutions $\psi$ to the nonlinear Schr\"odinger equation
	\begin{equation*}
		\left\{\begin{aligned}
			i\d_t\psi=&-\frac12\triangle\psi+f'(|\psi|^2)\psi\\
			\psi(0)=&\psi_0
		\end{aligned}\right.
	\end{equation*}
	then the functional consists in the square $L^2$ norm of $\d_t\psi$,
	\begin{equation*}
		I(t)=\int_{\R^d}|\d_t\psi|^2\,dx.
	\end{equation*}
	By using the polar factorization we see that it is possible to write
	\begin{equation*}
		I(t)=\int_{\R^d}(\d_t\sqrt{\rho})^2+\lambda^2\,dx,
	\end{equation*}
	where we define $\lambda:=\IM(\bar\phi\d_t\psi)$
	and direct computation shows 
	\begin{equation*}
		\xi=\sqrt{\rho}\lambda=\IM(\bar\psi\d_t\psi)=-\frac14\triangle\rho+\frac12|\nabla\sqrt{\rho}|^2+\frac12|\Lambda|^2+\rho f'(\rho)
	\end{equation*}
	Moreover by using the regularity of $\psi$ we also have that $\lambda=0$ a.e. in the vacuum set $\{\rho=0\}$. Therefore it is reasonable to define 
	\begin{equation*}
		\lambda=\left\{\begin{array}{cc}
			-\frac12\triangle\sqrt{\rho}+\frac12\frac{|\Lambda|^2}{\sqrt{\rho}}+f'(\rho)\sqrt{\rho}&\textrm{in }\;\{\rho>0\}\\
			0&\textrm{elsehwere}
		\end{array}\right.
	\end{equation*}
	For Schr\"odinger-generated solutions, $I(t)$ is uniformly bounded on compact time intervals due to the well-posedness results of NLS in $H^2(\R^d)$. Alternatively, here we give a direct calculation of the time derivative to $I(t)$ for smooth, non-vanishing solutions $(\sqrt\rho,\Lambda)$ to \eqref{eq:QHD_md}. To present this computation, we first state the following lemma, which shows the conservation law for the total energy density $e$ of smooth, non-vanishing solutions.
	
	\begin{lem}
		Let $(\rho, J)$ be a smooth solution to \eqref{eq:QHD_md} such that $\rho>0$. Then the energy density $e$ satisfies the following conservation law
		\begin{equation}\label{eq:en_cons}
			\d_te+\diver(\Lambda\lambda-\d_t\sqrt{\rho}\nabla\sqrt{\rho})=0.
		\end{equation}
	\end{lem}
	\begin{proof}
		Since the solution $(\rho,J)$ is smooth and $\rho>0$, the velocity field $v=J/\rho$ is well-defined, and we can write the system \eqref{eq:QHD_md} as
		\begin{equation*}
			\left\{\begin{aligned}
				&\d_t\rho+\diver(\rho v)=0\\
				&\rho\d_tv+\rho v\cdot\nabla v+\rho\nabla f'(\rho)
				=\frac12\rho\nabla\left(\frac{\triangle\sqrt{\rho}}{\sqrt{\rho}}\right).
			\end{aligned}\right.
		\end{equation*}
		For the equation of the momentum density, by using the chemical potential $\mu$ defined in \eqref{eq:chem_md}, it follows that
		\begin{equation*}
			\rho\d_tv+\rho\nabla\mu=0.
		\end{equation*}
		Analogously, the energy density reads
		\begin{equation*}
			e=\frac12|\nabla\sqrt{\rho}|^2+\frac12\rho |v|^2+f(\rho).
		\end{equation*}
		By using this expression we can differentiate it with respect to time and find
		\begin{equation*}
			\begin{aligned}
				\d_te=&\nabla\sqrt{\rho}\cdot\d_{t}\nabla\sqrt{\rho}+\left(\frac12|v|^2+f'(\rho)\right)\d_t\rho+\rho v\d_tv\\
				=&\diver\left(\nabla\sqrt{\rho}\d_t\sqrt{\rho}\right)
				+\left(-\frac12\frac{\triangle\sqrt{\rho}}{\sqrt{\rho}}+\frac12|v|^2+f'(\rho)\right)\d_t\rho+\rho v\d_tv.
			\end{aligned}
		\end{equation*}
		Last, by using the continuity equation of $\rho$ and definition \eqref{eq:chem_md} we then have
		\begin{equation*}\begin{aligned}
				\d_te=&\diver\left(\nabla\sqrt{\rho}\d_t\sqrt{\rho}\right)-\mu\diver(\rho v)-\rho v\cdot\nabla \mu\\
				=&\diver\left(\nabla\sqrt{\rho}\d_t\sqrt{\rho}-\rho v\mu\right).
			\end{aligned}
		\end{equation*}
	\end{proof}
	
	\begin{rem}
		The calculation in the lemma above requires smoothness and positivity of solutions in order to be rigorously justified. On the other hand it is an interesting question to see whether this conservation law, or its weaker version with the inequality, hold for a larger class of solutions. For this problem we refer the paper \cite{TzavFasc} where the authors determine some conditions on the velocity field and the mass density which allow to show the conservation of energy.
	\end{rem}
	
	\begin{prop}\label{prop:dtI_md}
		Let $(\rho, J)$ be a smooth solution to \eqref{eq:QHD_md} such that $\rho>0$. Then we have
		\begin{equation}\label{eq:dtI}
			\frac{d}{dt}I(t)=2\int_{\R^d}\lambda\d_t\sqrt\rho\,p'(\rho)\,dx.
		\end{equation}
	\end{prop}
	\begin{proof}
		For smooth, strictly positive solutions $(\rho,J)$ to \eqref{eq:QHD_md}, we can write our functional $I(t)$ as
		\begin{equation*}
			I(t)=\frac12\int_{\R^d}\rho(\mu^2+\sigma^2)\,dx,
		\end{equation*}
		where chemical potential 
		\begin{equation*}
			\mu=\lambda/\sqrt{\rho}=-\frac{\triangle\sqrt\rho}{2\sqrt\rho}+\frac12|v|^2+f'(\rho),
		\end{equation*}
		and we define $\sigma=\d_t\log\sqrt{\rho}$. In this framework we first  write down the evolution equations for $\mu, \sigma$. 
		
		By using the formula 
		\begin{equation*}
			\rho\mu=-\frac14\triangle\rho+e+p(\rho)
		\end{equation*}
		and the conservation law \eqref{eq:en_cons} of $e$, it follows that
		\begin{equation*}
			\begin{aligned}
				\d_t(\rho\mu)=&\d_t\left(e-\frac14\triangle\rho+p(\rho)\right)\\
				=&\diver(\nabla\sqrt{\rho}\d_t\sqrt{\rho}-\rho v\mu)-\frac14\d_{t}\triangle\rho+\d_tp(\rho).
			\end{aligned}
		\end{equation*}
		Again by using the continuity equation we may write
		\begin{equation}\label{eq:mu_evol}
			\rho\d_t\mu+\rho v\cdot\nabla\mu=\diver(\nabla\sqrt{\rho}\d_t\sqrt{\rho})-\frac14\d_{t}\triangle\rho+\d_tp(\rho).
		\end{equation}
		Now, to write the equation for $\sigma$ we may proceed in the following way. By writing the continuity equation as below
		\begin{equation*}
			\d_t\rho+v\cdot\nabla\rho+\rho\diver v=0,
		\end{equation*}
		we find the equation for $\log\sqrt{\rho}$, namely
		\begin{equation*}
			\d_t\log\sqrt{\rho}+v\cdot\nabla\log\sqrt{\rho}+\frac12\diver v=0.
		\end{equation*}
		We differentiate the last equation with respect to time and use $\d_tv=-\nabla\mu$, then 
		\begin{equation*}
			\d_t\sigma+v\cdot\nabla\sigma-\nabla\mu\cdot\nabla\log\sqrt{\rho}-\frac12\triangle\mu=0.
		\end{equation*}
		By multiplying this by $\rho$ we get
		\begin{equation}\label{eq:sigma_evol}
			\rho\d_t\sigma+\rho v\cdot\nabla\sigma=\frac12\diver(\rho\nabla\mu).
		\end{equation}
		Now we can use the equations \eqref{eq:mu_evol} and \eqref{eq:sigma_evol} to compute the time derivative of the functional $I(t)$. By differentiating the functional $I(t)$ in time using the continuity equation of $\rho$, the evolution \eqref{eq:mu_evol} and \eqref{eq:sigma_evol} of $\mu$ and $\sigma$, we obtain
		\begin{equation*}
			\begin{aligned}
				\frac{d}{dt}I(t)=&\int\frac12(\mu^2+\sigma^2)\d_t\rho+\mu\rho\d_t\mu+\sigma\rho\d_t\sigma\,dx\\
				=&\int\mu\left(\diver(\nabla\sqrt{\rho}\d_t\sqrt{\rho})-\frac14\triangle\d_{t}\rho+\d_tp(\rho))\right)\\
				&+\frac{\sigma}{2}\,\diver(\rho\nabla\mu)\,dx.
			\end{aligned}
		\end{equation*}
		Recall that $\sqrt{\rho}\sigma=\d_t\sqrt{\rho}$, $\rho\sigma=\frac12\d_t\rho$ and apply integration by parts, then we get
		\begin{equation*}
			\frac{d}{dt}I(t)=\int\mu\d_tp(\rho)\,dx=2\int\lambda\d_t\rho\,p'(\rho)\,dx.
		\end{equation*}
	\end{proof}
	
	Recall that $p(\rho)=\frac{\gamma-1}{\gamma}\rho^{\gamma}$, then as a direct consequence of identity \eqref{eq:dtI} we have 
	\begin{equation}\label{eq:115}
		\frac{d}{dt}I(t)=2(\gamma-1)\int_{\R^d}\rho^{\gamma-1}\lambda\d_t\sqrt\rho\,dx\lesssim \|\rho(t)\|_{L^\infty_{x}}^{\gamma-1}I(t).
	\end{equation}
	To apply the Gronwall's argument to functional $I(t)$, we need further control on $\|\rho(t)\|_{L^\infty_{x}}$, which can not be directly obtained from the energy bounds in multi-dimension. Here we recall the following 2D $\log-$type Sobolev inequality \cite{BG2}.
	
	\begin{lem}
		For any $f\in H^2(\R^2)$ we have 
		\begin{equation}\label{eq:logsob}
			\|f\|_{L^\infty(\R^2)}\le C(1+\sqrt{\log(1+\|f\|_{H^2(\R^2)})}).
		\end{equation}
	\end{lem}
	
	By using Proposition \ref{prop:dtI_md} and the $\log-$type Sobolev inequality in 2D, we obtain the uniform bound of functional $I(t)$ via standard Gronwall's argument.
	
	\begin{corollary}
		Let $(\rho, J)$ be a smooth solution to \eqref{eq:QHD_md} on $[0,T]\times\R^2$ such that $\rho>0$, and we further assume the pressure $p(\rho)=(1-\frac{1}{\gamma})\rho^\gamma$ with $1<\gamma\le \frac52$. Then for all $t\in[0,T]$ we have
		\[
		I(t)\leq C(T) I(0).
		\]
	\end{corollary}
	
	\begin{proof}
		Let us recall the time derivative \eqref{eq:dtI} of $I(t)$ and $f(\rho)=\rho\int_0^\rho\frac{p(s)}{s^2}\,ds$, then we have
		\[
		\frac{d}{dt}I(t)=2\int_{\R^2}\lambda\d_t\rho\,p'(\rho)dx=\int_{\R^2}\sqrt\rho\lambda\d_tf'(\rho)dx.
		\]
		By using the identity \eqref{eq:xi_def}, it follows
		\begin{align*}
			\frac{d}{dt}I(t)=&-\frac14\int_{\R^2}\triangle\rho\d_tf'(\rho)dx+\frac12\int_{\R^2}(|\nabla\sqrt\rho|^2+|\Lambda|^2)\d_tf'(\rho)dx\\
			&+\int_{\R^2}\rho f'(\rho)\d_tf'(\rho)dx=\textrm{I}+\textrm{II}+\textrm{III}.
		\end{align*}
		For the integral $\textrm{I}$, by integrating by parts and $f'(\rho)=\rho^{\gamma-1}$ we obtain
		\begin{equation}\label{eq:prop12_01}
			\textrm{I}=\frac{\gamma-1}{2}\frac{d}{dt}\int_{\R^2}\rho^{\gamma-1}|\nabla\sqrt\rho|^2dx+(\gamma-1)(\gamma-2)\int_{\R^2}\rho^{\gamma-\frac32}|\nabla\sqrt\rho|^2\d_t\sqrt\rho\,dx.
		\end{equation}
		The integral $\textrm{III}$ can be written as
		\[
		\textrm{III}=\frac{\gamma-1}{2\gamma-1}\frac{d}{dt}\int_{\R^2}\rho^{2\gamma-1}dx.
		\]
		To control the integral $\textrm{II}$, we will use Proposition \ref{prop:lift_pos} to associate the hydrodynamic functions $(\rho,J)(t)$ to a wave function $\psi(t,\cdot)\in H^2(\R^d)$ via polar factorization. However we remark that the $\psi(t,\cdot)$ is constructed at each time $t$ isolatedly and generally not a solution to any equation. By the polar factorization \eqref{eq:polar}, we have
		\begin{align*}
			\textrm{II}=\frac{\gamma-1}{2}\int_{\R^2}|\nabla\psi(t,x)|^2|\psi(t,x)|^{2\gamma-3}\d_t\sqrt\rho\,dx.
		\end{align*}
		Moreover by using Gagliardo-Nirenberg inequality and $\log-$type Sobolev inequality \eqref{eq:logsob}, it follows that
		\begin{align*}
			\textrm{II}\leq & C\|\nabla\psi(t,\cdot)\|_{L^4(\R^2)}^2\|\psi(t,\cdot)\|_{L^\infty}^{2\gamma-3}\|\d_t\sqrt\rho(t)\|_{L^2(\R^2)}\\
			\leq & C \|\nabla\psi(t,\cdot)\|_{L^2(\R^2)}\|\nabla^2\psi(t,\cdot)\|_{L^2(\R^2)}\\
			& (1+\sqrt{\log (1+\|\psi(t,\cdot)\|_{H^2(\R^2)})})^{2\gamma-3}\|\d_t\sqrt\rho(t)\|_{L^2(\R^2)}\\
			\leq & C\,I(t) \left(\log I(t)\right)^{\gamma-\frac32}+C(M(t),E(t)).
		\end{align*}
		Similarly the last term in the right hand side of \eqref{eq:prop12_01} satisfies the same estimate. Last, we have
		\[
		\int_{\R^2}\rho^{2\gamma-1}dx\leq C(\|\sqrt\rho(t)\|_{H^1(\R^2)})\leq C(M(t),E(t)),
		\]
		and 
		\begin{align*}
			\int_{\R^2}\rho^{\gamma-1}|\nabla\sqrt\rho|^2dx\leq &\|\psi(\cdot;t)^{\gamma-1}(t)\|_{L^2(\R^2)}\|\nabla\psi(\cdot;t)\|_{L^4(\R^2)}^2\\
			\leq & C(M(t),E(t))I(t)^\frac12\leq \frac12 I(t)+C(M(t),E(t)).
		\end{align*}
		Therefore we can conclude
		\[
		\frac{d}{dt}I(t)\leq C\, I(t)\left(\log I(t)\right)^{\gamma-\frac32}+C(M(t),E(t)),
		\]
		and the boundedness of $I(t)$ follows the conservation of $M(t)$, $E(t)$ and Gronwall's inequality for $\gamma-\frac32\leq 1$, namely $\gamma\leq \frac52$.
	\end{proof}
	
	\begin{rem}
		The Gronwall inequality used in the proof above clearly restricts the boundedness result of $I(t)$ to $\gamma\leq \frac52$. The case $\gamma > \frac52$ in 2D and the higher dimensional cases would require a finer analysis, for instance by using some dispersive type estimates. This is consistent for example with what is discussed in \cite{OV,PTV}.
	\end{rem}
	
	\section{Stability of weak solutions with positive density}\label{sect:comp_md}
	
	In the last section of this paper on multi-dimensional QHD system, we will consider GCP weak solutions as defined in Definition \ref{def:cptsln}, in the framework of finite energy and functional $I(t)$. 
	More precisely, let us now consider a sequence of weak solutions $\{(\sqrt\rho_n,\Lambda_n)\}$ to \eqref{eq:QHD_md}, which satisfies the following uniform bounds,
	
	\begin{equation}\label{eq:unif1_md}
		\begin{aligned}
			&\|\sqrt{\rho_n}\|_{L^\infty(0, T;H^1(\R^d))}+\|\Lambda_n\|_{L^\infty(0, T;L^2(\R^d))}\leq M_1\\
			&\|\d_t\sqrt{\rho_n}\|_{L^\infty(0, T;L^2(\R^d))}+\|\lambda_n\|_{L^\infty(0, T;L^2(\R^d))}\leq M_2,
		\end{aligned}
	\end{equation}
	where the generalized chemical potential $\lambda_n$ is such that 
	\begin{equation*}
		\sqrt{\rho_n}\lambda_n=-\frac14\triangle\rho_n + e_n + p(\rho)
	\end{equation*}
	and 
	\begin{equation*}
		e_n=\frac12|\nabla\sqrt{\rho_n}|^2+\frac12|\Lambda_n|^2+f(\rho_n).
	\end{equation*}
	When the density is strictly positive, $\lambda_n$ can be explicitly written as 
	\begin{equation}\label{eq:lambda_n_md}
		\lambda_n=\frac12\left(-\triangle\sqrt\rho_n+\frac{|\Lambda_n|^2}{\sqrt{\rho_n}}\right)+f'(\rho_n)\sqrt\rho_n.
	\end{equation}
	To obtain the stability result, we need further information on the energy density $e_n$ besides the natural energy bound, which depends on some additional assumptions on the sequence of solutions as stated in Theorem \ref{thm:stab_md}. In the framework of these assumptions, we can show that the energy density satisfies an $H^1$ estimate, which play essential roles in our stability argument. We remark that those conditions are also considered in \cite{AMZ} for the 1-dimensional QHD system. However in the 1-dimensional case, a Morawetz-type estimate which provides the $L^2_{t,x}$ integrability of the energy density can be obtained by considering a entropy inequality, and the analogue of such estimate in the multi-dimensional case is still missing.
	
	\begin{prop}\label{prop:unif2_md}
		Let $\{(\sqrt{\rho_n},\Lambda_n)\}$ be a sequence of finite energy weak solutions to the QHD system \eqref{eq:QHD_md} satisfying \eqref{eq:unif1_md}. 
		Moreover for all $n\in\N$ and a.e. $t\in[0,T]$, we assume one of the following condition is satisfied: 
		\begin{itemize}
			\item[(1)] $\rho_n(t,\cdot)$ is continuous with $\rho_n(t,\cdot)>0$;
			\item[(2)] $(\rho_n, J_n)$ is a spherically symmetric solution, and the energy density $e_n=\frac12|\nabla\sqrt\rho_n|^2+\frac12|\Lambda_n|^2+f(\rho_n)$ is continuous.
		\end{itemize}
		Then we have 
		\begin{equation}\label{eq:unif2_md}
			\|\triangle\rho_n\|_{L^\infty_t L^2_x}+\|\nabla J_n\|_{L^\infty_t L^2_x}+\|\nabla \sqrt{e_n}\|_{L^\infty_t L^2_x}\leq C(M_1,M_2).
		\end{equation}
	\end{prop}

	\begin{proof}
		We first consider the case (1) of positive mass density. To obtain the estimates \eqref{eq:unif2_md}, we notice that for a.e. $t\in[0,T]$, the hydrodynamic state $(\sqrt{\rho_n},\Lambda_n)(t,\cdot)$ satisfies the assumption of Proposition \ref{prop:lift_pos}, hence it can be lifted to a wave function $\psi_n(t,\cdot)\in H^2_x(\R^d)$ with uniform bounds
		\[
		\|\psi_n\|_{L^\infty_t H^2_x}\le C(M_1,M_2).
		\] 
		
		The estimates \eqref{eq:unif2_md} is a direct consequence of the $L^\infty_t H^2_x$ bound of $\psi_n$. By using $\triangle\rho_n=2\RE(\bar\psi_n\triangle\psi_n)+2|\nabla\psi_n|^2$, H\"older's inequality and Sobolev embedding we have
		\begin{equation*}
			\|\triangle\rho_n\|_{L^\infty_t L^2)}\lesssim\|\psi_n\|^2_{L^\infty_t H^2_x}\leq C(M_1,M_2),
		\end{equation*}
		On the other hand, by the Madelung transformation
		\begin{equation*}
			\|\nabla J_n\|_{L^\infty_tL^2_x}=\|\nabla\IM(\bar\psi_n\nabla\psi_n)\|_{L^\infty_tL^2_x}\leq\|\psi_n\|_{L^\infty_tL^2_x}^2\leq C(M_1,M_2).
		\end{equation*}
		Finally, let us recall that by the polar factorization we have 
		\[
		e_n=\frac12|\nabla\sqrt{\rho_n}|^2+\frac12|\Lambda_n|^2+f(\rho_n)=\frac12|\nabla\psi_n|^2+f(|\psi_n|^2).
		\]
		It is straightforward to see that 
		$|\nabla|\nabla\psi_n||\leq|\nabla^2\psi_n|$ a.e. $x$, so that
		\begin{equation*}
			\|\nabla\sqrt{e_n}\|_{L^\infty_tL^2_x}\leq\|\psi_n\|_{L^\infty_tH^2_x}\leq C(M_1,M_2).
		\end{equation*}
		
		Similarly, if we assume the condition (2) , we can apply Proposition \ref{prop:lift2_s} to the hydrodynamic state $(\sqrt{\rho_n},\Lambda_n)(t,\cdot)$ to obtain a spherically symmetric associated wave function $\psi_n(t,\cdot)\in H^2_x(\R^d)$. Thus the estimates \eqref{eq:unif2_md} follows the same argument as before.
	\end{proof}
	
	The bounds in \eqref{eq:unif1_md} imply that, up to passing to subsequences, we have
	\begin{align}\label{eq:weak1_md}
		\sqrt{\rho_n}\rightharpoonup&\sqrt{\rho},\qquad L^\infty(0, T;H^1(\R^d))\cap W^{1, \infty}(0, T;L^2(\R^d)),\\
		\label{eq:weak2_md}
		\Lambda_n\rightharpoonup&\Lambda,\qquad L^\infty(0, T;L^2(\R^d)),
	\end{align}
	and the uniform bound of $e_n$ in \eqref{eq:unif2_md} implies local strong convergence 
	\begin{equation}\label{eq:locstrong1_md}
		\sqrt{e_n}\to\omega,\qquad L^2(0, T;L^2_{loc}(\R^d)).
	\end{equation} 
	it is not straightforward to show the identity for the limit $\omega$ and $(\sqrt{\rho},\Lambda)$, because of the lack of compactness for $(\nabla\sqrt{\rho_n}, \Lambda_n)$. However the next proposition shows $(\sqrt{\rho_n},\Lambda_n)$ indeed converges strongly. Furthermore, we prove that $\omega$ and $\Lambda$ vanish almost everywhere in the vacuum $\{\rho=0\}$, which matches the physical interpretation of the energy density.
	
	\begin{prop}\label{prop:strong2_md}
		Let $\{(\sqrt{\rho_n},\Lambda_n)\}$ be a sequence of solutions to the QHD system \eqref{eq:QHD_md} in the GCP class, satisfying the assumptions of Proposition \ref{prop:unif2_md}. Then we have
		\begin{equation}
			\omega^2=\frac12|\nabla\sqrt{\rho}|^2+\frac12|\Lambda|^2+f(\rho)
		\end{equation}
		is satisfied for a.e. $(t,x)\in[0\,T]\times\R^d$, and $\omega=0$ a.e. $(t,x)\in\{\rho=0\}$. Furthermore we have the following local strong convergence
		\begin{equation}
			\begin{aligned}
				\nabla\sqrt{\rho_n}&\to \nabla\sqrt{\rho},&L^2(0, T;L^2_{loc}(\R^d)),\\
				\Lambda_n&\to\Lambda,&L^2(0, T;L^2_{loc}(\R^d)).
			\end{aligned}
		\end{equation}
	\end{prop}
	
	\begin{proof}
		We first multiply the energy density by $\rho_n$,
		\begin{equation}\label{eq:112}
			\rho_n e_n=\frac18|\nabla\rho_n|^2+\frac12|J_n|^2+f(\rho_n)\rho_n.
		\end{equation}
		By using the uniform bounds \eqref{eq:unif2_md} we have that (up to subsequences),
		\begin{equation*}
			\begin{aligned}
				\rho_n\to&\rho, \quad L^2(0,T;H^1_{loc}(\R^d))\\
				J_n\to&J,\quad L^2(0,T;L^2_{loc}(\R^d)),
			\end{aligned}
		\end{equation*}
		hence, by passing to the limit in \eqref{eq:112} we obtain
		\begin{equation}\label{eq:113}
			\rho\omega^2=\frac18|\nabla\rho|^2+\frac12|J|^2+f(\rho)\rho,\quad\textrm{in}\;L^1_{t,x,loc}.
		\end{equation}
		On the other hand, by passing to the limit in $J_n=\sqrt{\rho_n}\Lambda_n$, it follows that 
		\[
		J=\sqrt{\rho}\Lambda,\quad \textrm{in}\;L^1_{t,x,loc},
		\]
		which shows that \eqref{eq:113} can be equivalently written as 
		\begin{equation}\label{eq:114}
			\rho\left(\omega^2-\frac12|\nabla\sqrt{\rho}|^2-\frac12|\Lambda|^2-f(\rho)\right)=0.
		\end{equation}
		We now claim that $\omega=0$ a.e. on $\{\rho=0\}$. By \eqref{eq:lambda_n_md} we have
		\begin{equation*}
			e_n=\sqrt{\rho_n}\lambda_n+\frac14\triangle\rho_n-p(\rho_n)
		\end{equation*}
		and as $n\to\infty$ it follows that
		\begin{equation*}
			\omega^2=\sqrt{\rho}\lambda+\frac14\triangle\rho-p(\rho)\quad \textrm{in}\;L^1_{t,x,loc}.
		\end{equation*}
		By using Lemma \ref{lemma:LL}, we have $\triangle\rho=0$ a.e. on $\{\rho=0\}$, consequently $\omega=0$ a.e. on $\{\rho=0\}$.
		On the other hand, for any $0<R<\infty$ let us define $V_R=\{\rho=0\}\cap  (B_R(0)\times[0,T))$, then by using Fatou's Lemma and the local strong convergence of $\rho_n$, $\sqrt{e_n}$, we have
		\begin{align*}
			\int_0^T\int_{V_R}\frac12|\nabla\sqrt{\rho}|^2+\frac12|\Lambda|^2\,dxdt\leq&\liminf_{n\to\infty}\int_0^T\int_{V_R}\frac12|\nabla\sqrt{\rho_n}|^2+\frac12|\Lambda_n|^2\,dxdt\\
			=&\lim_{n\to\infty}\int_0^T\int_{V_R}e_n-f(\rho_n)\,dxdt\\
			=&\int_0^T\int_{V_R}\omega^2-f(\rho)\,dxdt=0.
		\end{align*}
		This implies that $\Lambda=0$ a.e. on $\{\rho=0\}$ and consequently from \eqref{eq:113} we also have 
		\begin{equation*}
			\omega^2=\frac12|\nabla\sqrt{\rho}|^2+\frac12|\Lambda|^2+f(\rho).
		\end{equation*}
		Last, to prove local strong convergence, as in the previous argument we have that for any $R>0$,
		\begin{align*}
			\int_0^T\int_{B_R(0)}\frac12|\nabla\sqrt{\rho}|^2+\frac12|\Lambda|^2\,dxdt&\leq\liminf_{n\to\infty}\int_0^T\int_{B_R(0)}\frac12|\nabla\sqrt{\rho_n}|^2+\frac12|\Lambda_n|^2\,dxdt\\
			&=\lim_{n\to\infty}\int_0^T\int_{B_R(0)} e_n-f(\rho_n)\,dxdt\\
			&=\int_0^T\int_{B_R(0)}\omega^2-f(\rho)\,dxdt\\
			&=\int_0^T\int_{B_R(0)}\frac12|\nabla\sqrt{\rho}|^2+\frac12|\Lambda|^2\,dxdt,
		\end{align*}
		and we can conclude 
		\begin{equation*}
			\|\nabla\sqrt{\rho}\|_{L_{t}^{2}L_{x, loc}^{2}}=\lim_{n\to\infty}\|\nabla\sqrt{\rho_n}\|_{L_{t}^{2}L_{x, loc}^{2}},\quad\|\Lambda\|_{L_{t}^{2}L_{x, loc}^{2}}=\lim_{n\to\infty}\|\Lambda_n\|_{L_{t}^{2}L_{x, loc}^{2}}.
		\end{equation*}
		Then by standard argument we can reinforce the weak convergence to a strong one.
	\end{proof}
	
	To conclude this section, we prove the stability Theorem \ref{thm:stab_md} of weak solutions $(\sqrt{\rho_n},\Lambda_n)$ with positive density, which is a direct consequence of the compactness Proposition \ref{prop:strong2_md}. As we remarked in the introduction that we do not require any uniform lower bound on the density, and the limiting weak solution may contain vacuum.
	
	\begin{proof}[Proof of Theorem \ref{thm:stab_md}]
		By the definition of weak solutions, sequence $\{(\rho_{n},J_{n})\}$
		satisfy, for any $n\ge1$, test function $\eta\in \mathcal{C}_{c}^\infty([0,T]\times\R^d)$ and vector test function $\xi\in \mathcal{C}^\infty_c([0,T]\times\R^d;\R^d)$
		\[
		\int_{0}^{T}\int_{\R^d}\rho_{n}\partial_{t}\eta+J_{n}\cdot\nabla\eta dxdt+\int_{\R^d}\rho_{n,0}\eta(0)dx=0,
		\]
		\begin{align*}
			\int_{0}^{T}\int_{\R^d}[J_{n}\cdot\partial_{t}\xi+(\Lambda_{n}\otimes\Lambda_n+
			\nabla\sqrt{\rho_{n}}\otimes\nabla\sqrt{\rho_n}):\nabla\xi\\
			p(\rho_n)\diver \xi-\frac{1}{4}\rho_{n}\triangle\diver\xi]dxdt+\int_{\R^d}J_{n,0}\cdot\xi(0)dx & =0.
		\end{align*}
		
		Then by Proposition \ref{prop:strong2_md}, the strong convergence of $\rho_{n}$
		and $J_{n}$ in $L_{t}^{2}L_{x, loc}^{2}$ allows us to pass to the limit in the continuity equation,
		\[
		\int_{0}^{T}\int_{\R^d}\rho\partial_{t}\eta+J\cdot\nabla\eta dxdt+\int_{\R^d}\rho_0\eta(0)dx=0,
		\]
		On the other hand, by using the local strong convergence of $\nabla\sqrt{\rho_n}$ and $\Lambda_n$ in $L^2_tL^2_{x,loc}$, we can pass to the limit in the weak formulation of the equation for momentum density  in order to obtain
		\begin{align*}
			\int_{0}^{T}\int_{\R^d}[J\cdot\partial_{t}\xi+(\Lambda\otimes\Lambda+
			\nabla\sqrt{\rho}\otimes\nabla\sqrt{\rho}):\nabla\xi\\
			p(\rho)\diver \xi-\frac{1}{4}\rho\triangle\diver\xi]dxdt+\int_{\R^d}J_{0}& \cdot\xi(0)dx=0.
		\end{align*}
		
		Last we need to check the generalized irrotationality condition is preserved. For weak solutions $(\rho_n,J_n)$, by Definition \ref{def:FEWS} we have 
		\[
		\nabla\wedge J_n=2\nabla\sqrt{\rho_n}\wedge\Lambda_n.
		\]
		Using the compactness of $J_n$, $\nabla\sqrt\rho_n$ and $\Lambda_n$ and passing to the limit, we conclude 
		\[
		\nabla\wedge J=2\nabla\sqrt{\rho}\wedge\Lambda,\quad a.e.\ (t,x)\in [0,T]\times\R^d.
		\]
		Therefore by definition $(\rho,J)$ is a weak solution to the Cauchy problem of the
		QHD system.
	\end{proof}

\section*{Acknowledgements}
The third author acknowledges Gran Sasso Science Institute, where most of this work was done as part of his PhD Thesis.
	

\end{document}